\documentclass[pdflatex,sn-mathphys-num]{sn-jnl}% Math and Physical Sciences Numbered Reference Style 
\usepackage{graphicx}%
\usepackage{multirow}%
\usepackage{amsmath,amssymb,amsfonts}%
\usepackage{amsthm}%
\usepackage{mathrsfs}%
\usepackage[title]{appendix}%
\usepackage{xcolor}%
\usepackage{textcomp}%
\usepackage{manyfoot}%
\usepackage{booktabs}%
\usepackage{algorithm}%
\usepackage{algorithmicx}%
\usepackage{algpseudocode}%
\usepackage{listings}%
\usepackage[all]{xy}
\usepackage{enumitem}
\theoremstyle{thmstyleone}%
\newtheorem{theorem}{Theorem}%  meant for continuous numbers
\newtheorem{proposition}[theorem]{Proposition}% 

\newtheorem{corollary}[theorem]{Corollary}
\theoremstyle{thmstyletwo}%
\newtheorem{example}{Example}%
\newtheorem{remark}{Remark}%

\theoremstyle{thmstylethree}%
\newtheorem{definition}{Definition}%

\def\br#1\er{\textcolor{red}{#1}}

\newcommand{\Ten}{\mathcal{T}}
\newcommand{\Fun}{\mathcal{F}}
\newcommand{\Can}{\mathbb{C}}
\newcommand{\vd}{\dot{\partial}}
\newcommand{\Ker}{\mathrm{Ker}}
\newcommand{\Img}{\mathrm{Img}}
\newcommand{\dd}{\mathrm{d}}
\newcommand{\TT}{\mathrm{T}}

\newcommand{\wt}{\widetilde}

\newcommand{\y}{y}
\newcommand{\HH}{\mathrm{H}}
\newcommand{\VV}{\mathrm{V}}

\begin{document}

\title[Ladder of Finsler-type objects on spacetimes]{The ladder of Finsler-type objects and their variational problems on spacetimes}


\author[1]{\fnm{Miguel} \sur{Sánchez}}\email{sanchezm@ugr.es}

\author*[1]{\fnm{Fidel F.} \sur{Villaseñor}}\email{fidelfv@ugr.es}
\equalcont{These authors contributed equally to this work.}


\affil[1]{\orgdiv{Departamento de Geometría y Topología / IMAG}, \orgname{Universidad de Granada}, \orgaddress{\street{Av. de la Funete Nueva s/n}, \postcode{18071 Granada}, \country{Spain}}}

%\affil[2]{\orgdiv{Department}, \orgname{Organization}, \orgaddress{\street{Street}, \city{City}, \postcode{10587}, \state{State}, \country{Country}}}
%
%\affil[3]{\orgdiv{Department}, \orgname{Organization}, \orgaddress{\street{Street}, \city{City}, \postcode{610101}, \state{State}, \country{Country}}}


\abstract{The space of anisotropic $r$-contravariant $s$-covariant  
	$\alpha$-homogeneous tensors on a manifold admits a functorial structure where vertical derivatives $\dot \partial$ and contractions $\imath_{\Can}$ by the Liouville  vector field $\mathbb{C}$ are operators which maintain $s+\alpha$ constant. In  (semi-)Finsler   geometry,   this structure is transmitted faithfully to connection-type elements yielding   the following   ladder: geodesic sprays / nonlinear connections / anisotropic connections / linear (Finslerian) connections. However, it is more loosely transmitted to metric-type ones: Finslerian Lagrangians / Legendre transformations / anisotropic metrics. 
	
	We will study this structure in depth and apply it to discuss the recent variational proposals (Einstein-Hilbert, Einstein-Palatini, Einstein-Cartan) for  generalizing  Einstein equations to the Finsler  setting. }




\maketitle

\tableofcontents

\section{Introduction} \label{introduction}

It is well known that a (2-homogeneous) semi-Finsler Lagrangian\footnote{Consistently with the extension of the name \emph{Riemannian} to  \emph{semi-Riemannian} \cite{ON}, we will use \emph{semi-Finsler} to stress   regularity  of the Lagrangian (nondegeneracy of its vertical Hessian), and \emph{pseudo-Finsler} when this condition is dropped. Contrary to some of our previous works \cite{JSV1,JSV2}, here we do not use \emph{metric} to refer to the function $L$, since we want to clarify the difference with the anisotropic metrics of Definition \ref{anis_met}.  }  $L$ (with $L=F^2$ at least in the positive definite case) yields as a series of structures with increasing complexity: a geodesic spray $G$ associated with the curves that are extremal for the energy functional, a nonlinear connection $N$ 
determined by $G$ and a set of   linear (Finsler)   connections (Berwald, Chern, Cartan, Hashiguchi...) constructed on the   vertical   bundle by using $N$ and other elements of $L$. This Lagrangian also determines a set of anisotropic connections, each one, $\nabla$, illustrating the intuitive idea of having an affine connection $\nabla^V$ for the oriented directions determined by a nonvanishing vector field $V$. Anisotropic connections   go back at least to   \cite{Mts} but they do not belong to the mainstream Finslerian setting. Recently, they were studied systematically by M. Á. Javaloyes \cite{Jav19} and they were identified  as {\em vertically trivial} Finsler connections in \cite{JSV1} (a contribution to the previous Lorentzian meeting).  

This permits to construct a {\em ladder of connection-type  structures}:
\begin{quote}
	linear (Finsler) con. $\hat{\nabla}$  $\hookleftarrow$
	anisotropic con. $\nabla$
	$\hookleftarrow$ 
	nonlinear con. $N$ $\hookleftarrow$
	geodesic spray $G$.
\end{quote} 
In this ladder, each step has an infinite dimensional affine structure on a vector space of anisotropic tensors. Using the canonical coordinates $(x,y)$ of the tangent bundle, each level also has has a natural expression with a  cocycle transformation  as well as a natural degree of (positive) homogeneity in $y$. The step at its right increases in one the order  of  homogeneity and  decreases in one its order of covariance.    (Except for the transition between $\hat{\nabla}$ and $\nabla$, which, here and below, works in a different but formally analogous way.)    %(see Table \ref{table} below).

It is possible to move from any step to the next one at the left by vertically differentiating, which implies a canonical choice in the new step.  It is also possible to move to the step at the  right by contracting with the Liouville vector field $\mathbb{C}$, which implies  dropping  some information. This dropped information is the difference between an element of the ladder and the canonical one obtained by moving once to the right and once to the left. This difference  will be called {\em residual}, but this name does not mean that such an information is not relevant.   Indeed, for example, the residue of a nonlinear connection is its torsion, and the difference between the Chern and Berwald anisotropic connections, i.e. the Landsberg tensor, 
%(or between the analog Finsler connections) 
 is also residual.     

All the previous objects (sprays and different sorts of connections) can be defined with independence of any Lagrangian $L$. However, $L$ also suggests a   potential   {\em ladder of metric-type structures}, 
%\begin{quote} 
%semi-Riemannian   $g_x$  
%$\hookrightarrow$ fundamental tensor  
%$g^L_{(x,[y])} $   $\hookrightarrow  $ anisotropic metric  $g_{(x,[y])}$
%\end{quote}    
\begin{quote} 
	anisotropic metric  $g(=g_{(x,[y])})$
	$\hookleftarrow$ Legendre transformation  $\ell$
	$\hookleftarrow  $ 	Finsler Lagrangian  $L$ %$g_x$  
\end{quote}    
 Recall that a semi-Riemannian metric  $g=g_x$, $x\in M$, is a nondegenerate scalar product depending smoothly on $x$; this includes  Riemannian (positive definite) and Lorentzian (index one) metrics.   Meanwhile,   an anisotropic 
metric $g=g_{(x,[y])}$ is  a nondegenerate  scalar product which depends smoothly not only on the 
point $x$, but also on the oriented direction $[y]=\mathbb{R}^+\cdot  y$. (See Definitions \ref{Legendre} and \ref{anis_met} about our notions related to Legendre transformations.)
%fundamental tensor $g^L$ of the pseudo-Finsler Lagrangian $L$ is defined as a vertical Hessian, see \ref{}. \footnote{Tal vez para poner en algún sitio: $g^L$ admits the following geometric interpretation if $[y]$ is not a degenerate direction. At each point $x_0$, consider the indicatrix $L(x_0,y)=\pm 1$, then  $g^L(x_0,[y])$ is the second fundamental form of the indicatrix at the direction $[y]$ (computed affinely with respect to the own direction position vector, which is transverse to the indicatrix, see Figure \ref{fig}).}
In this ladder, one can move to the left   almost   trivially; however, there will not be a natural way to move to the right, at least if the  nondegeneracy of the metrics is to be preserved.

In the present article,  
functorial transitions between the different geometric objects involved in the  semi-Finsler setting will be analyzed in detail. As an application, we will discuss the different variational approaches to properly generalize the Einstein equations. 

We start by considering the ladder of  $\alpha$-homogeneous, $s$-covariant, $r$-contravariant tensors in Section \ref{section tensors}. The contraction with the Lioville vector $\imath_{\Can}$ and the vertical  derivative  $\dot\partial$ are natural (functorial)  transformations that preserve $\alpha+s$ and satisfy a simple rule (Proposition \ref{main prop}), which permits the construction of the ladder (Definition~\ref{def ladder}) and residues (Definition \ref{residues}) as well as the systematic computation of the latter (Theorem \ref{main th}, Remark \ref{synthesis}).

In Section \ref{section metrics}, we   analyze the situation for the metric-type tensors. We point out that symmetry and regularity conditions need not be preserved when moving on the ladder labeled by $(r=0,\omega=2)$. With a detailed example, we also see that even if the nondegeneracy condition is respected, the signature of the metrics may change (Example \ref{example regularity}). This prevents us from constructing a ladder with transitions as satisfactory of those of general tensors.  

In Section \ref{section connections},  we analyze the ladder for connection-type objects which live in a conic subset $A$ of $\TT M$ (geodesic spray, nonlinear connection, anisotropic conection), which correspond to the case $(r=1, \omega=2)$.   We do so by extending $\imath_{\Can}$ and $\vd$ to act on these objects, by means of a systematic procedure that highlights their transformation cocycles (Proposition \ref{prop connections} and Corollary \ref{corollary connections}).   It is worth pointing out that, when the elements of this ladder are associated  with a Finsler   Lagrangian   $L$, the natural inclusions starting at the   corresponding   spray yield the canonical nonlinear connection and the Berwald anisotropic one.   (The Chern connection would appear in a different way, for example by defining a new inclusion $\{$ nonlinear con.  $\}\hookrightarrow_L$ $ \{ $ anisotropic con. $\}$ that makes explicit use of the Landsberg tensor.)   %However, there is a second  inclusion naturally associated with $L$, $\{$ nonlinear con.  $\}\hookrightarrow_L$ $ \{ $ anisotropic connection $\}$ which yields the Chern anisotropic and, then, the Landsberg tensor $\hbox{Lan}$, as aforementioned.   

In Section \ref{section linear},   the aforementioned   ladder is extended to linear connections   that   live on the vertical bundle $\VV A \subset \VV(\TT M)$   (or, equiv., on the pullback bundle $\pi_A^*(\TT M)$).   This last step is subtler because such a connection $\hat \nabla$ does  not project directly onto an anisotropic one, but onto a nonlinear   one $N=N^{\hat{\nabla}}$ when a regularity condition is imposed.     Now, all the linear connections  projecting on the same $N$ admit a natural decomposition into an anisotropic connection $\nabla$ and a residue (Theorem \ref{cor linear intrinsic}), thus completing the ladder. 

Finally, in Section \ref{variational problems}, we discuss the different variational approaches to generalize (vacuum) Einstein equations considered so far.   This includes the Einstein-Hilbert one by Pfeifer, Wohlfart, Hohmann \& Voicu (PWHV) \cite{PW,HPV}, the Einstein-Palatini one by Javaloyes, S\'anchez  $\&$  Villaseñor (JSV) \cite{JSV2} and the anisotropic metric one by García-Parrado   $\&$ Minguzzi  (GM) \cite{GPMin}.  
The variational approach requires a metric-type object to define a volume element and, eventually, carry out contractions. Then, to construct the action functional, one  must choose the involved level of the   (extended)   connection-type ladder and whether it is the one associated with the metric object or not.   PWHV consider the  Lorentz-Finsler Lagrangian   $L$ to compute volumes as well as to obtain the nonlinear connection $N$ and the Ricci scalar. The equation is the variational completion of   $\mathrm{Ric}^L=0$ and depends explicitly on  the Landsberg tensor $\mathrm{Lan}$. The    JSV approach considers $N$ as independent on $L$ and, then, states equations for both $N$ and $L$. Even though this gives more freedom and variety of solutions, strong uniqueness results for $N$ are obtained under mild conditions. However,   the canonical connection of $L$    is shown  to be {\em not} a solution of the metric-affine equations when $\mathrm{Lan}\neq 0$.   The GM   approach uses the highest level $\hat \nabla$ to define the action.   In order to enable systematic comparisons between these and other theories, we state Theorem \ref{functionals} and Remark \ref{functionals otro}. These results allow one to take any Lagrangian density for one kind of connection-type objects and naturally redefine it to make sense for any other kind.  Thus, the choice of levels and variables in the semi-Finsler ladder yields a much bigger variety of possibilities than in the semi-Riemannian case,  yielding dramatically different theories of gravity. We illustrate this with GM's results, where the kind of objects and variations is so demanding   that even a non-quadratic  Lorentzian norm cannot be a vacuum solution.

\section{The ladder structure of anisotropic tensors} \label{section tensors}

\subsection{Preliminaries and conventions} \label{preliminaries}

First, we introduce notions and notation standard in semi-Finsler geometry and, more generally, anisotropic tensor calculus; see e.g. \cite{Jav19} for background. Thus, let $M$ be a smooth $n$-dimensional manifold\footnote{We take this to be Hausdorff and second countable, e.g.   in order to have the existence of globally defined sections of the affine bundles of \S \ref{section connections}.   By \emph{smooth}, we mean $\mathcal{C}^m$ with $m$ large enough that all the derivatives that we consider of anisotropic tensors exist and are continuous on the set $A$. Al objects will be assumed to be smooth unless stated otherwise.} and $A\subseteq\TT M\setminus\mathbf{0}$ be an open subset all of whose fibers $A_p:=A\cap\TT_p M$ (for $p\in M$) are nonempty and conic (i.e., if $v\in A_p$, then $\lambda v\in A_p$ for all $\lambda\in\mathbb{R}^+$). 

As in previous occasions \cite{JSV1,JSV2}, we shall work with $A$-anisotropic tensor fields on $M$, or, for short, \emph{anisotropic tensors}, writing $\Ten^r_s(M_A)$ for the set of all of these of type $(r,s)$, where $r,s\in\mathbb{N}\cup\left\{0\right\}$. Recall that 
\[
\Ten^r_s(M_A)=\left\{\text{sections of }\pi^*_A(\TT M\otimes\overset{r)}{\ldots}\otimes\TT M\otimes \TT^\ast M\otimes\stackrel{s)}{\ldots}\otimes\TT^\ast M)\longrightarrow A\right\},
\]
where $\pi_A$ is the restriction to $A$ of the natural projection $\TT M\rightarrow M$. The readers more familiar with the vertical bundle formalism \cite[\S 4.1.3]{KLS} may think that $\VV A\equiv\pi^*_A(\TT M)$ as bundles over $A$, via the isomorphism given fiberwise by 
\begin{equation}
	\TT_{\pi(v)}M\ni w\equiv \left.\frac{\dd}{\dd t}(v+tw)\right|_0\in\VV_v A.
	\label{vertical isomorphism}
\end{equation}
This way, $\pi^*_A(\TT M\otimes\overset{r)}{\ldots}\otimes\TT M\otimes \TT^\ast M\otimes\stackrel{s)}{\ldots}\otimes\TT^\ast M)$ is naturally a tensor product of copies of $\VV A$ and its dual. When one has the additional datum of a complementary horizontal bundle $\HH A\subset\TT A$, there is also an isomorphism $\HH A\equiv\pi^*_A(\TT M)$, but we wish to make our constructions independent of any kind of connection. (For the possibilities for this, see \cite{JSV1}, whose results we will refine in \S \ref{section connections}.) This is also the reason for not considering d-tensors \cite[\S 2.5]{BM}. We shall use the special notation $\Fun(A):=\Ten^0_0(M_A)$ for the set of functions defined on $A$ and consider $\Ten^r_s(M)\subset\Ten^r_s(M_A)$ by identifying each section $T\colon M\rightarrow \TT M\otimes\overset{r)}{\ldots}\otimes\TT M\otimes \TT^\ast M\otimes\stackrel{s)}{\ldots}\otimes\TT^\ast M$ with $T\circ\pi_A\in\Ten^r_s(M_A)$.

Now, let $\alpha\in\mathbb{R}$ be given. By $\mathrm{h}_\alpha\Ten^r_s(M_A)\subset\Ten^r_s(M_A)$, we will denote the set of those anisotropic tensors that are (positively) homogeneous of degree $\alpha$, or, for short, \emph{$\alpha$-homogeneous}. That is (always in natural coordinates $\left(x^i,\y^i\right)$ associated with arbitrary coordinates $\left(x^i\right)$ on $M$, and with the Einstein convention in the Latin indices), those 
\[
T=T^{i_1\ldots i_r}_{j_1\ldots j_s}\,\partial_{x^{i_1}}\otimes\ldots\otimes\partial_{x^{i_r}}\otimes\dd x^{j_1}\otimes\ldots\dd x^{j_s}\in\Ten^r_s(M_A)
\]
 whose components satisfy that, whenever $\lambda\in\mathbb{R}^+$,
\[
T^{i_1\ldots i_r}_{j_1\ldots j_s}(x,\lambda\y)=\lambda^\alpha\,T^{i_1\ldots i_r}_{j_1\ldots j_s}(x,\y).
\]
(Writing $\mathrm{h}_\alpha\Ten^r_s(M_A)$ for the set of all these $T$'s and not $\mathrm{h}^\alpha\Ten^r_s(M_A)$, as in \cite{JSV2}, is intentional; we will comment on the intuition for this in the third item of Rem. \ref{synthesis}.) As the main example, the Liouville (or canonical) anisotropic vector field is 
\[
\Can\in\mathrm{h}_1\Ten^1_0(M_A),\qquad\Can_v:=v
\]
for $v\in A$ (indeed, $\Can=\Can^i\,\partial_{x^i}=\y^i\,\partial_{x^i}$). The vertical derivative operator acting on $T\in\Ten^r_s(M_A)$ is given by 
\[
\vd T=T^{i_1\ldots i_r}_{j_1\ldots j_s\,\cdot k}\,\partial_{x^{i_1}}\otimes\ldots\otimes\partial_{x^{i_r}}\otimes\dd x^{j_1}\otimes\ldots\dd x^{j_s}\otimes\dd x^k\in\Ten^r_{s+1}(M_A).
\]
Here, we have agreed that the index of derivation will be the last covariant one of the resulting tensor and we have introduced the standard notation
\[
T^{i_1\ldots i_r}_{j_1\ldots j_s\,\cdot k}:=\frac{\partial T^{i_1\ldots i_r}_{j_1\ldots j_s}}{\partial\y^k}.
\]
 In the same vein, let us define the operator $\imath_{\Can}$ to be the interior product (contraction) with the canonical field of any $S\in\Ten^r_{s+1}(M_A)$ on its last index:
\[
\imath_{\Can} S:=S(-,\ldots,-,\Can)= S^{i_1\ldots i_r}_{j_1\ldots j_s a}\y^a\,\partial_{x^{i_1}}\otimes\ldots\otimes\partial_{x^{i_r}}\otimes\dd x^{j_1}\otimes\ldots\dd x^{j_s}\in\Ten^r_{s}(M_A).
\]
Just by inspection of the components, one sees that if $T\in\mathrm{h}_\alpha\Ten^r_s(M_A)$, then $\vd T\in\mathrm{h}_{\alpha-1}\Ten^r_{s+1}(M_A)$, and if $S\in\mathrm{h}_{\alpha-1}\Ten^r_{s+1}(M_A)$, then $\imath_{\Can} S\in\mathrm{h}_\alpha\Ten^r_s(M_A)$.
Finally, Euler's well-known theorem on homogenous functions \cite[Th. 1.2.1]{BCS} will play a key role. In our notation, it states that $T\in\mathrm{h}_\alpha\Ten^r_s(M_A)$ if and only if 
\begin{equation}
	T^{i_1\ldots i_r}_{j_1\ldots j_s\,\cdot a}\y^a=\alpha\,T^{i_1\ldots i_r}_{j_1\ldots j_s}.
	\label{euler}
\end{equation}

\begin{remark} \label{conventions}
	By the above comments, the restrictions
	\[
	\left.\imath_{\Can}\right|_{\mathrm{h}_{\alpha-1}\Ten_{s+1}^r(M_A)}\colon \mathrm{h}_{\alpha-1}\Ten_{s+1}^r(M_A)\longrightarrow \mathrm{h}_{\alpha}\Ten_s^r(M_A),
	\]
	\[
	\left.\vd\right|_{\mathrm{h}_{\alpha}\Ten_{s}^r(M_A)}\colon\mathrm{h}_{\alpha}\Ten_s^r(M_A)\longrightarrow\mathrm{h}_{\alpha-1}\Ten_{s+1}^r(M_A)
	\]
	are well-defined. Whenever it is convenient and the tensors' type is understood from the context, we shall use the abbreviations
	\[
	\overset{\alpha}{\imath_{\Can}}:=\left.\imath_{\Can}\right|_{\mathrm{h}_{\alpha-1}\Ten_{s+1}^r(M_A)}, \qquad \underset{\alpha}{\vd}:=\left.\vd\right|_{\mathrm{h}_{\alpha}\Ten_{s}^r(M_A)}. 
	\]
	In any case, we stress our convention that the two operators act on the last argument of any $S\in\mathrm{h}_{\alpha-1}\Ten_{s+1}^r(M_A)$ or $T\in\mathrm{h}_{\alpha}\Ten_s^r(M_A)$, respectively.
\end{remark}

\subsection{Construction of the ladder} \label{section ladder}

The basic result for our setup is the following.

\begin{proposition} \label{main prop}
	For each $\alpha\in\mathbb{R}$, one has that\footnote{Recall that there is nothing of the sort of Einstein summation holding for $\alpha$, only for the indices $i$, $j$, $k$... Still, we find the notations $\overset{\alpha}{\imath_{\Can}}$ and $\underset{\alpha}{\vd}$ to be the most suggestive ones for the algebraic structure that we are introducing on $\underset{r,s}{\bigoplus}\underset{\alpha}{\bigoplus}\mathrm{h}_\alpha\Ten^r_s(M_A)$, see e.g. \eqref{ladder}.}
	\begin{equation}
		\overset{\alpha}{\imath_{\Can}}\circ\underset{\alpha}{\vd}=\alpha\,\mathrm{Id}.
		\label{alg_euler}
	\end{equation}
	As consequences, if $\alpha\neq 0$, then  $\overset{\alpha}{\imath_{\Can}}\colon\mathrm{h}_{\alpha-1}\Ten^r_{s+1}(M_A)\rightarrow \mathrm{h}_{\alpha}\Ten^r_{s}(M_A)$ is surjective, $\underset{\alpha}{\vd}\colon\mathrm{h}_{\alpha}\Ten^r_{s}(M_A)\rightarrow\mathrm{h}_{\alpha-1}\Ten^r_{s+1}(M_A)$ is injective and 
	\begin{equation}
		\mathrm{h}_{\alpha-1}\Ten^r_{s+1}(M_A)=\Img(\underset{\alpha}{\vd})\oplus\Ker(\overset{\alpha}{\imath_{\Can}}),
		\label{decomposition}
	\end{equation}
	with the corresponding projections being given by
	\begin{equation}
		\mathrm{h}_{\alpha-1}\Ten^r_{s+1}(M_A)\longrightarrow\Img(\left.\vd\right|_{\mathrm{h}_{\alpha}\Ten^r_{s}(M_A)}), \qquad S\longmapsto\vd(\frac{1}{\alpha}\,\imath_{\Can}S);
		\label{proj img}
	\end{equation}
	\begin{equation}
	\mathrm{h}_{\alpha-1}\Ten^r_{s+1}(M_A)\longrightarrow\Ker(\left.\imath_{\Can}\right|_{\mathrm{h}_{\alpha-1}\Ten^r_{s+1}(M_A)}), \qquad S\longmapsto S-\vd(\frac{1}{\alpha}\,\imath_{\Can}S).
	\label{proj ker}
	\end{equation}
\end{proposition}
\begin{proof}
	The identity \eqref{alg_euler} is just Euler's theorem \eqref{euler} taking into account Rem. \ref{conventions}, namely our conventions for $\overset{\alpha}{\imath_{\Can}}$ and $\underset{\alpha}{\vd}$. Hence, when $\alpha\neq 0$, these are one-sided inverses of each other up to a multiplicative constant, so the former must be surjective, and the latter, injective. Moreover, by composing \eqref{alg_euler} with $\overset{\alpha}{\imath_{\Can}}$ on the right, one obtains that
	\begin{equation}
		\overset{\alpha}{\imath_{\Can}}\circ\left( \alpha\,\mathrm{Id}-\underset{\alpha}{\vd}\circ\overset{\alpha}{\imath_{\Can}}\right)=0.
		\label{alg_euler conseq}
	\end{equation}
	
	For $S\in	\mathrm{h}_{\alpha-1}\Ten^r_{s+1}(M_A)$, suppose that $S=\vd T+R$, where $T\in\mathrm{h}_{\alpha}\Ten^r_{s}(M_A)$ and $R\in\Ker(\left.\imath_{\Can}\right|_{\mathrm{h}_{\alpha-1}\Ten^r_{s+1}(M_A)})$. Then, applying \eqref{alg_euler} to $S$,
	\[
	\overset{\alpha}{\imath_{\Can}} S=\overset{\alpha}{\imath_{\Can}}\underset{\alpha}{\vd} T+\overset{\alpha}{\imath_{\Can}}R=\alpha T,
	\]
	which shows that the components of $S$ in $\Img(\left.\vd\right|_{\mathrm{h}_{\alpha}\Ten^r_{s}(M_A)})$ and $\Ker(\left.\imath_{\Can}\right|_{\mathrm{h}_{\alpha-1}\Ten^r_{s+1}(M_A)})$ must necessarily be given by \eqref{proj img} and \eqref{proj ker}, resp. Now, trivially, indeed $\vd(\frac{1}{\alpha}\,\imath_{\Can}S)\in\Img(\left.\vd\right|_{\mathrm{h}_{\alpha}\Ten^r_{s}(M_A)})$ and $S=\vd(\frac{1}{\alpha}\,\imath_{\Can}S)+\left\{S-\vd(\frac{1}{\alpha}\,\imath_{\Can}S)\right\}$; that $S-\vd(\frac{1}{\alpha}\,\imath_{\Can}S)\in\Ker(\left.\imath_{\Can}\right|_{\mathrm{h}_{\alpha-1}\Ten^r_{s+1}(M_A)})$ follows from \eqref{alg_euler conseq}.
\end{proof}

There is quite a visual way of organizing the information obtained from successively applying Prop. \ref{main prop}:
\begin{definition} \label{def ladder}
	Let $\omega\in\mathbb{N}\cup\left\{0\right\}$. The \emph{ladder of $A$-anisotropic tensors labeled by $(r,\omega)$} is the double sequence of maps
	\begin{equation}
		\xymatrix{ 
			  \mathrm{h}_{0}\Ten^r_\omega\ar@/^/[r]^{\;\;\overset{1}{\imath_{\Can}}} & \;\ldots\;\ar@/^/[l]^{\;\;\underset{1}{\vd}}\ar@/^/[r] & \mathrm{h}_{\omega-2}\Ten^r_2\ar@/^/[l]\ar@/^/[r]^{\overset{\omega-1}{\imath_{\Can}}} &  \mathrm{h}_{\omega-1}\Ten^r_1\ar@/^/[l]^{\underset{\omega-1}{\vd}}\ar@/^/[r]^{\overset{\omega}{\imath_{\Can}}} & \mathrm{h}_{\omega}\Ten^r_0\ar@/^/[l]^{\;\;\underset{\omega}{\vd}}.}
		\label{ladder}
	\end{equation}
\end{definition}
\begin{remark}
	\,
	\begin{enumerate}
		\item We have omitted the $M_A$ from the notation of \eqref{ladder} in order to stress its functorial and natural character. Anyway, we shall not delve in these matters, so in all our reasonings we think of the fibered manifold $A\rightarrow M$ as fixed.
		
		\item By construction, $\overset{0}{\imath_{\Can}}$ and $\underset{0}{\vd}$ do not appear in our ladders, so each $\overset{\alpha}{\imath_{\Can}}$ in \eqref{ladder} is an epimorphism, and each $\underset{\alpha}{\vd}$, a monomorphism, say, of $\Fun(M)$-algebras (as $\underset{\alpha}{\vd}$ would be Leibnizian and not linear for functions defined on $A$).
		
		\item Had we allowed $\omega\in\mathbb{R}\setminus\left(\mathbb{N}\cup\left\{0\right\}\right)$, then the ladder would have been prolonged infinitely to the left:
		\[
		\xymatrix{ 
			 \;\ldots\;\ar@/^/[r]^{\overset{\omega-s}{\imath_{\Can}}\;\quad}&\mathrm{h}_{\omega-s}\Ten^r_s\ar@/^/[l]^{\underset{\omega-s}{\vd}\quad}\ar@/^/[r]^{\quad\overset{\omega-s+1}{\imath_{\Can}}} & \;\ldots\;\ar@/^/[l]^{\quad\underset{\omega-s+1}{\vd}}\ar@/^/[r] & \mathrm{h}_{\omega-2}\Ten^r_2\ar@/^/[l]\ar@/^/[r]^{\overset{\omega-1}{\imath_{\Can}}} & \mathrm{h}_{\omega-1}\Ten^r_1\ar@/^/[l]^{\underset{\omega-1}{\vd}}\ar@/^/[r]^{\;\;\overset{\omega}{\imath_{\Can}}} &  \mathrm{h}_{\omega}\Ten^r_0. \ar@/^/[l]^{\;\;\underset{\omega}{\vd}}  }
		\]
		This case keeps the properties of Prop. \ref{main prop} at all \emph{levels} $s\in\mathbb{N}\cup\left\{0\right\}$ (as $\omega-s\neq 0$). However, it does not come up in applications. On the other end, the only sensible way to prolong the ladder to the right appears to be as $0$.

		\item The main point of the diagram \eqref{ladder} is  its lack of commutativity insofar as $\underset{\alpha}{\vd}\circ\overset{\alpha}{\imath_{\Can}}$ is not the identity. It is precisely the failure of $\underset{\alpha}{\vd}\circ\overset{\alpha}{\imath_{\Can}}$ to be $\alpha\,\mathrm{Id}\left(=\overset{\alpha}{\imath_{\Can}}\circ\underset{\alpha}{\vd}\right)$ what produces a nontrivial $\Ker(\overset{\alpha}{\imath_{\Can}})$, see \eqref{proj ker}. This makes \eqref{ladder} differ from a mere sequence of isomorphisms.
	\end{enumerate}
\end{remark}

Let us discuss our intuitions on the ladder labeled by $(r,\omega)$ with $\omega\geq 1$. On it, $\mathrm{h}_\omega\Ten^r_0$ appears as the \emph{ground floor} (or \emph{level $0$}) and contains the objects with the highest homogeneity degree and lowest number of indices. One can regard these as the simplest objects, due to the surjection of the \emph{next floor} (or \emph{level $1$}) $\overset{\omega}{\imath_{\Can}}\colon\mathrm{h}_{\omega-1}\Ten^r_1\rightarrow\mathrm{h}_{\omega}\Ten^r_0$. Indeed, as $\underset{\omega}{\vd}\colon\mathrm{h}_{\omega}\Ten^r_0\rightarrow\mathrm{h}_{\omega-1}\Ten^r_1$ is an injection, \eqref{decomposition} gives
\[
\begin{matrix}
	\mathrm{h}_{\omega-1}\Ten^r_1&=&\Img(\underset{\omega}{\vd})\oplus\Ker(\overset{\omega}{\imath_{\Can}})&\equiv&\mathrm{h}_{\omega}\Ten^r_0\times\Ker(\overset{\omega}{\imath_{\Can}}), \\
	\underset{1}{S}&=&\vd\underset{0}{S}+\underset{1}{\Delta}&\equiv&(\underset{0}{S},\underset{1}{\Delta}).
\end{matrix}
\]
Moreover, under this $\Fun(M)$-algebra isomorphism, $\overset{\omega}{\imath_{\Can}}$ becomes essentially the trivial projection. Indeed, due to \eqref{alg_euler},   
\[
(\underset{0}{S},\underset{1}{\Delta})\equiv\vd\underset{0}{S}+\underset{1}{\Delta}\stackrel{\overset{\omega}{\imath_{\Can}}}{\longmapsto}\omega\,\underset{0}{S}.
\]
Of course, this procedure can be iterated:
\[
\begin{matrix}
	\mathrm{h}_{\omega-2}\Ten^r_2&=&\Img(\underset{\omega-1}{\vd})\oplus\Ker(\overset{\omega-1}{\imath_{\Can}})&\equiv&\mathrm{h}_{\omega-1}\Ten^r_1\times\Ker(\overset{\omega-1}{\imath_{\Can}})&\equiv&\mathrm{h}_{\omega}\Ten^r_0\times\Ker(\overset{\omega}{\imath_{\Can}})\times\Ker(\overset{\omega-1}{\imath_{\Can}}), \\
	 \underset{2}{S}&=&\vd\underset{1}{S}+\underset{2}{\Delta}&\equiv & (\underset{1}{S},\underset{2}{\Delta})&\equiv & (\underset{0}{S},\underset{1}{\Delta},\underset{2}{\Delta});
\end{matrix}
\]
\[
(\underset{0}{S},\underset{1}{\Delta},\underset{2}{\Delta})\equiv\vd\vd\underset{0}{S}+\vd\underset{1}{\Delta}+\underset{2}{\Delta}\stackrel{\overset{\omega-1}{\imath_{\Can}}}{\longmapsto}\left(\omega-1\right)\left(\vd\underset{0}{S}+\underset{1}{\Delta}\right)\equiv\left(\omega-1\right)(\underset{0}{S},\underset{1}{\Delta})\stackrel{\overset{\omega}{\imath_{\Can}}}{\longmapsto}\omega\left(\omega-1\right)\underset{0}{S},
\]
and so on, until obtaining
\[
\begin{matrix}
	\mathrm{h}_{0}\Ten^r_{\omega}&=&\Img(\underset{1}{\vd})\oplus\Ker(\overset{1}{\imath_{\Can}})&\equiv&\mathrm{h}_1\Ten^r_{\omega-1}\times\Ker(\overset{1}{\imath_{\Can}})&\equiv&\mathrm{h}_{\omega}\Ten^r_0\times\Ker(\overset{\omega}{\imath_{\Can}})\times\ldots\times\Ker(\overset{1}{\imath_{\Can}}), \\
	\underset{\omega}{S}&=&\vd\underset{\omega-1}{S}+\underset{\omega}{\Delta}&\equiv & (\underset{\omega-1}{S},\underset{\omega}{\Delta})&\equiv & (\underset{0}{S},\underset{1}{\Delta},\ldots,\underset{\omega}{\Delta});
\end{matrix}
\]
\[
(\underset{0}{S},\underset{1}{\Delta},\ldots,\underset{\omega}{\Delta})\stackrel{\overset{1}{\imath_{\Can}}}{\longmapsto}(\underset{0}{S},\underset{1}{\Delta},\ldots,\underset{\omega-1}{\Delta})\stackrel{\overset{2}{\imath_{\Can}}}{\longmapsto}2(\underset{0}{S},\underset{1}{\Delta},\ldots,\underset{\omega-2}{\Delta})\stackrel{\overset{3}{\imath_{\Can}}}{\longmapsto}\ldots\stackrel{\overset{\omega}{\imath_{\Can}}}{\longmapsto}\omega!\,\underset{0}{S}.
\]
Let us state a general version of this result, in case that one does not want to start at the ground floor or finish at the \emph{top floor} (or \emph{level $\omega$}), which is $\mathrm{h}_{0}\Ten^r_{\omega}$.
\begin{theorem} \label{main th}
	Let $\omega\in\mathbb{N}\cup\left\{0\right\}$ and also $\alpha,\beta\in\left\{0,\ldots,\omega\right\}$ with $\alpha < \beta$. Then, there is a canonical isomorphism of $\Fun(M)$-algebras
	\[
	\begin{matrix}
		&\mathrm{h}_{\alpha}\Ten^r_{\omega-\alpha} &\equiv&\mathrm{h}_{\beta}\Ten^r_{\omega-\beta}\times\Ker(\overset{\beta}{\imath_{\Can}})\times\ldots\times\Ker(\overset{\alpha+2}{\imath_{\Can}})\times\Ker(\overset{\alpha+1}{\imath_{\Can}}), \\
		&\underset{\omega-\alpha}{S}&\equiv&(\underset{\omega-\beta}{S},\underset{\omega-\beta+1}{\Delta},\ldots,\underset{\omega-\alpha-1}{\Delta},\underset{\omega-\alpha}{\Delta}).
	\end{matrix}
	\]
	Under it, the transitions between the levels $s=\omega-\alpha$ and $s=\omega-\beta$ of the ladder \eqref{ladder} are given by
	\begin{equation}
		\begin{matrix}
			\underset{\omega-\beta}{S}\stackrel{\underset{\beta}{\vd}}{\longmapsto}(\underset{\omega-\beta}{S},0),\quad(\underset{\omega-\beta}{S},\underset{\omega-\beta+1}{\Delta})\stackrel{\underset{\beta-1}{\vd}}{\longmapsto}(\underset{\omega-\beta}{S},\underset{\omega-\beta+1}{\Delta},0),\quad\ldots, \\
			(\underset{\omega-\beta}{S},\underset{\omega-\beta+1}{\Delta},\ldots,\underset{\omega-\alpha-1}{\Delta})\stackrel{\underset{\alpha+1}{\vd}}{\longmapsto}(\underset{\omega-\beta}{S},\underset{\omega-\beta+1}{\Delta},\ldots,\underset{\omega-\alpha-1}{\Delta},0);
		\end{matrix}
	\label{sequence injections}
	\end{equation}

	\begin{equation}
		\begin{matrix}
			(\underset{\omega-\beta}{S},\ldots,\underset{\omega-\alpha-1}{\Delta},\underset{\omega-\alpha}{\Delta})\stackrel{\overset{\alpha+1}{\imath_{\Can}}}{\longmapsto}\left(\alpha+1\right)(\underset{\omega-\beta}{S},\ldots,\underset{\omega-\alpha-1}{\Delta})\stackrel{\overset{\alpha+2}{\imath_{\Can}}}{\longmapsto}\ldots \\
			\stackrel{\overset{\beta-1}{\imath_{\Can}}}{\longmapsto}\left(\displaystyle\prod_{\nu=\alpha+1}^{\beta-1}\nu\right)(\underset{\omega-\beta}{S},\underset{\omega-\beta+1}{\Delta})\stackrel{\overset{\beta}{\imath_{\Can}}}{\longmapsto}\left(\displaystyle\prod_{\nu=\alpha+1}^{\beta}\nu\right)\underset{\omega-\beta}{S}.
		\end{matrix}
		\label{sequence projections}
	\end{equation}
\end{theorem}
\begin{proof}
	One reproduces the above process, now starting at $\mathrm{h}_{\beta-1}\Ten^r_{\omega-\beta+1}\equiv\mathrm{h}_{\beta}\Ten^r_{\omega-\beta}\times\Ker(\overset{\beta}{\imath_{\Can}})$ so that $\underset{\beta}{\vd}\colon\underset{\omega-\beta}{S}\in\mathrm{h}_{\beta}\Ten^r_{\omega-\beta}\mapsto(\underset{\omega-\beta}{S},0)\in\mathrm{h}_{\beta}\Ten^r_{\omega-\beta}\times\Ker(\overset{\beta}{\imath_{\Can}})$ and  $\overset{\beta}{\imath_{\Can}}\colon(\underset{\omega-\beta}{S},\underset{\omega-\beta+1}{\Delta})\in\mathrm{h}_{\beta}\Ten^r_{\omega-\beta}\times\Ker(\overset{\beta}{\imath_{\Can}})\mapsto\beta\underset{\omega-\beta}{S}\in\mathrm{h}_{\beta}\Ten^r_{\omega-\beta}$. In general, one uses \eqref{decomposition} to decompose $\mathrm{h}_{\nu-1}\Ten^r_{\omega-\nu+1}\equiv\mathrm{h}_{\nu}\Ten^r_{\omega-\nu}\times\Ker(\overset{\nu}{\imath_{\Can}})$ iterating from $\nu=\beta$ to $\nu=\alpha+1$.
	
	At each step, the injectivity of $\underset{\nu}{\vd}$ is being used in establishing \eqref{sequence injections}, while one uses \eqref{alg_euler} to establish \eqref{sequence projections}. Still, for further clarification, let us use this chance to write down explicitly the components of a general element of each of $\mathrm{h}_{\beta-1}\Ten^r_{\omega-\beta+1}$, ..., $\mathrm{h}_{\alpha}\Ten^r_{\omega-\alpha}$:
	\[
	\underset{\omega-\beta+1}{S}\in\mathrm{h}_{\beta-1}\Ten^r_{\omega-\beta+1}\implies	\underset{\omega-\beta+1}{S}^I_{j_1\ldots j_{\omega-\beta+1}}=\underset{\omega-\beta}{S}^I_{j_1\ldots j_{\omega-\beta}\, \cdot j_{\omega-\beta+1}}+\underset{\omega-\beta+1}{\Delta}^I_{j_1\ldots j_{\omega-\beta+1}},
	\]
	
	\[
	\underset{\omega-\beta+2}{S}\in\mathrm{h}_{\beta-2}\Ten^r_{\omega-\beta+2}\implies\underset{\omega-\beta+2}{S}^I_{j_1\ldots j_{\omega-\beta+2}}=\underset{\omega-\beta+1}{S}^I_{j_1\ldots j_{\omega-\beta+1}\,\cdot j_{\omega-\beta+2}}+\underset{\omega-\beta+2}{\Delta}^I_{j_1\ldots j_{\omega-\beta+2}}
	\]
	\[
	=\underset{\omega-\beta}{S}^I_{j_1\ldots  \cdot j_{\omega-\beta+1}\cdot j_{\omega-\beta+2}}+\underset{\omega-\beta+1}{\Delta}^I_{j_1\ldots j_{\omega-\beta+1}\,\cdot j_{\omega-\beta+2}}+\underset{\omega-\beta+2}{\Delta}^I_{j_1\ldots j_{\omega-\beta+2}},
	\]
	
	\[
	\vdots
	\]
	
	\[
	\begin{split}
		\underset{\omega-\alpha}{S}\in\mathrm{h}_{\alpha}\Ten^r_{\omega-\alpha}\implies\underset{\omega-\alpha}{S}^I_{j_1\ldots j_{\omega-\alpha}}=&\underset{\omega-\beta}{S}^I_{j_1\ldots j_{\omega-\beta}\, \cdot j_{\omega-\beta+1}\ldots\cdot j_{\omega-\alpha}} \\
		&+\sum_{\nu=\alpha+1}^{\beta}\underset{\omega-\nu+1}{\Delta}^I_{j_1\ldots j_{\omega-\nu+1}\,\cdot j_{\omega-\nu+2}\ldots\cdot j_{\omega-\alpha}}.
	\end{split}
	\]
	Here we have abbreviated $I=(i_1,\ldots, i_r)$ and the homogeneity degree of $\underset{\omega-\nu}{S}^I_{j_1\ldots j_{\omega-\nu}}$ is $\nu$, while that of $\underset{\omega-\nu+1}{\Delta}^I_{j_1\ldots j_{\omega-\nu+1}}$ is $\nu-1$ with $\underset{\omega-\nu+1}{\Delta}^I_{j_1\ldots j_{\omega-\nu}\,a}y^a=0$. With these properties and Euler's theorem \eqref{euler}, it is straightforward to check \eqref{sequence projections}.
\end{proof}
All in all, we see that, when $\beta-\alpha\in\mathbb{N}$, an element of $\mathrm{h}_{\alpha}\Ten^r_{\omega-\alpha}$ is the same as one of $\mathrm{h}_{\beta}\Ten^r_{\omega-\beta}$ together with a tuple $(\underset{\omega-\beta+1}{\Delta},\ldots,\underset{\omega-\alpha-1}{\Delta},\underset{\omega-\alpha}{\Delta})\in\Ker(\overset{\beta}{\imath_{\Can}})\times\ldots\times\Ker(\overset{\alpha+2}{\imath_{\Can}})\times\Ker(\overset{\alpha+1}{\imath_{\Can}})$ which may or may not be $0$. We have just expressed the precise sense in which the elements of $\mathrm{h}_{\alpha}\Ten^r_{\omega-\alpha}$ are more complex than those of $\mathrm{h}_{\beta}\Ten^r_{\omega-\beta}$. This goes all the way to the level $\omega$ (namely $\mathrm{h}_{0}\Ten^r_{\omega}$), which contains the most general objects. Due to \eqref{sequence injections}, if an element of $\underset{\omega-\alpha}{S}\in\mathrm{h}_{\alpha}\Ten^r_{\omega-\alpha}$ has, say, $\underset{\omega-\alpha}{\Delta}=0$, this means that it is essentially an element of $\mathrm{h}_{\alpha+1}\Ten^r_{\omega-\alpha-1}$, and if it has $(\underset{1}{\Delta},\underset{2}{\Delta},\ldots,\underset{\omega-\alpha}{\Delta})=(0,\ldots,0)$, then it essentially belongs to $\mathrm{h}_{\omega}\Ten^r_{0}$. Notice that, by construction, the $\underset{s}{\Delta}$ appear as obstructions to certain   integrabilities,   which justifies the following terminology:
\begin{definition} \label{residues}
	Given $\underset{\omega-\alpha}{S}\in\mathrm{h}_{\alpha}\Ten^r_{\omega-\alpha}$, we call the corresponding $\underset{\omega-\beta+1}{\Delta}$, ..., $\underset{\omega-\alpha-1}{\Delta}$,  $\underset{\omega-\alpha}{\Delta}$ defined in Th. \ref{main th} the \emph{(integrability) residues} of $\underset{\omega-\alpha}{S}$ with respect to the level $\omega-\beta$ of the ladder \eqref{ladder}.

\end{definition}
\begin{remark} \label{synthesis}
	As a synthesis of all the possible transitions on \eqref{ladder}, as expressed by \eqref{sequence injections} and \eqref{sequence projections}, one has that:
	\begin{itemize}
		\item Going up on the ladder and then back down (i.e., applying $\imath_{\Can}\circ\ldots\circ\imath_{\Can}\circ\vd\circ\ldots\circ\vd$) leaves each object intact up to a constant factor.
		
		\item But doing the opposite (i.e., applying $\vd\circ\ldots\circ\vd\circ\imath_{\Can}\circ\ldots\circ\imath_{\Can}$) destroys their residues! When going down to level $0$, the ground floor, the resulting map is 
		\begin{equation}
			\begin{split}
				\underset{\omega-\alpha}{S}\equiv(\underset{0}{S},\underset{1}{\Delta},\ldots,\underset{\omega-\alpha}{\Delta})\longmapsto\left(\prod_{\nu=\alpha+1}^{\omega}\nu\right)(\underset{0}{S},0,\ldots,0)&\equiv\left(\prod_{\nu=\alpha+1}^{\omega}\nu\right)\underset{\alpha+1}{\vd}\ldots\underset{\omega}{\vd}\underset{0}{S} \\
				&=\left(\prod_{\nu=\alpha+1}^{\omega}\nu\right)\vd^{\omega-\alpha}\underset{0}{S}.
			\end{split}
		\label{destroying defects}
		\end{equation}
		In other words, the result will always be an element of $\mathrm{h}_{\omega}\Ten_0^r$, identifiable with another element of $\mathrm{h}_{\alpha}\Ten_{\omega-\alpha}^r$ (its $\left(\omega-\alpha\right)$-th vertical derivative). In general, $	\underset{\omega-\alpha}{S}$ and $\vd^{\omega-\alpha}\underset{0}{S}$ will differ and each of them may be regarded as a ``correction" of the other, depending on one's purpose. (For example, in \S \ref{section connections}, there will be a correspondence of $\underset{\omega-\alpha}{S}$ and $\vd^{\omega-\alpha}\underset{0}{S}$ with, resp., the Chern and Berwald anisotropic connections of a semi-Finsler Lagrangian, so the residues will amount to the Landsberg tensor.)
		
		\item The operator $\imath_{\Can}$ subtracts an index of the tensor on which it acts to add a degree of homogeneity whereas $\vd$ does the opposite, but this one does not destroy information. Thus, the sum of the homogeneity degree and the number of covariant indices always remains equal to $\omega$, recall \eqref{ladder}. This is the motivation for the notation $\mathrm{h}_{\alpha}\Ten_s^r$ instead of $\mathrm{h}^{\alpha}\Ten_s^r$: in a precise sense, the number $\alpha$ counts ``hidden" covariant indices.
		
		\item From \S \ref{preliminaries}, we have chosen $\imath_{\Can}$ and $\vd$ to act on the last index, and our concrete construction of \eqref{ladder}, Prop. \ref{main prop} and Th. \ref{main th} depend on this convention. But the map \eqref{destroying defects} does not! So, if, say, we would have chosen $\imath_{\Can}$ and $\vd$ to act on the first index, then, for $\underset{\omega-\alpha}{S}\in\mathrm{h}_{\alpha}\Ten^r_{\omega-\alpha}$, the individual residues $\underset{1}{\Delta}$, ..., $\underset{\omega-\alpha}{\Delta}$ would change. But, what would remain unchanged is whether they are $0$ or not, and also the underlying $\underset{0}{S}\in\mathrm{h}_{\omega}\Ten^r_{0}$.
	\end{itemize} 
\end{remark}
%Even though it will not be on our focus of attention, for completeness, let us comment on the singular case $\alpha=0$. Then, \eqref{alg_euler} states that $\Img(\left.\vd\right|_{\mathrm{h}_{\alpha}\Ten^r_{s}(M_A)})\subseteq\Ker(\left.\imath_{\Can}\right|_{\mathrm{h}_{\alpha-1}\Ten^r_{s+1}(M_A)})$.
%\begin{itemize}
%	\item The interior product $\overset{0}{\imath_{\Can}}$ is still surjective. This can be seen by taking a Riemannian metric $h$ on $M$ and, for each $R\in\mathrm{h}_{0}\Ten^r_{s}(M_A)$, considering $R\otimes\frac{h(\Can,-)}{h(\Can,\Can)}\in \mathrm{h}_{-1}\Ten^r_{s+1}(M_A)$, which satisfies
%	\[
%	\overset{0}{\imath_{\Can}}(R\otimes\frac{h(\Can,-)}{h(\Can,\Can)})=R\,\frac{h(\Can,\Can)}{h(\Can,\Can)}=R.
%	\]
%	
%	\item The vertical derivative $\underset{0}{\vd}$ is not injective and, assuming that the fibers $A\cap\TT_pM$ are connected, its kernel consists precisely on the classical tensor fields on $M$. Indeed, for $R\in\mathrm{h}_{0}\Ten^r_{s}(M_A)$, that $\vd R=0$ means that $R$ is (locally) constant on each fiber; thus
%	\[
%	\Ker(\left.\vd\right|_{\mathrm{h}_{0}\Ten^r_{s}(M_A)})=\Ten^r_{s}(M).
%	\]
%\end{itemize}

In the next two sections, we study the application of this ladder structure to the two cases of interest in semi-Finsler geometry and its generalizations:
\begin{enumerate}
	\item   \emph{Metric-type objects}, by which we mean semi-Finsler Lagrangians or anisotropic metrics, appear on the ladder labeled by $(r,\omega)=(0,2)$. (Therefore, Legendre transformations will be included in a natural way.)   Still, these elements are not just any kind of tensors. Rather, they fulfill symmetry and nondegeneracy conditions that must be carefully taken into account to complement their situation on the ladder.
	
	\item   \emph{Connection-type objects}, such as nonlinear or anisotropic connections (but also sprays),   will appear in a new ladder of affine spaces directed by the vector ones labeled by $(r,\omega)=(1,2)$.    Furthermore, the linear connections on the vertical bundle $\VV A\rightarrow A$ will be included in \S \ref{section linear} as a special kind of prolongation   of this affine ladder.
\end{enumerate}
  Later, in \S \ref{variational problems}, we shall discuss the implications of this structure on the functionals definable for each type of object. Indeed, due to our developments, some functionals for elements in   one level of the ladder can be defined in a natural way on the other levels. This is important for the various extensions of general relativity to Lorentz-Finsler geometry and beyond, which are of a great interest nowadays.

\section{Applications to metric-type objects} \label{section metrics}

\subsection{From semi-Finsler Lagrangians to anisotropic metrics} \label{section metrics 1}

The starting point of semi-Finsler geometry is a $2$-homogeneous function defined on $A$ \cite{BF}. Therefore, this object belongs to the ground floor of \eqref{ladder} for $(r,\omega)=(0,2)$. Here, we give its definition and interpretations of its first and second vertical derivatives.   This is standard, and only names may change with respect to previous accounts, e.g. \emph{semi-Finsler Lagrangian} instead of \emph{pseudo-Finsler metric}.   However, we will \emph{not} employ the notation $g$ for its vertical Hessian; instead, we will reserve $g$ to denote an arbitrary anisotropic metric (Def. \ref{anis_met}).
\begin{definition}
	A function $L\in\mathrm{h}_2\mathcal{F}(A)=\mathrm{h}_2\Ten_0^0(M_A)$ is called a \emph{semi-Finsler Lagrangian} provided that the covariant anisotropic tensor $\vd^2L$ is nondegenerate on all of $A$. Then, the \emph{Legendre transformation associated with $L$} is 
	\[
	\vd L\in\mathrm{h}_1\Ten_1^0(M_A),
	\]
	while its \emph{fundamental tensor} is 
	\begin{equation}
		\frac{1}{2}\vd^2 L\in\mathrm{h}_0\Ten_2^0(M_A).
		\label{fundamental tensor}
	\end{equation}
 We shall denote the set of all semi-Finsler Lagrangians by $\mathscr{M}_{\mathrm{s-F}}(A)$.
\end{definition}
Our notion of associated Legendre transformation agrees with Minguzzi's \cite[Def. 4]{Min}, and with Dahl's \cite[Def. 1.8]{Dahl} up to a factor of $\frac{1}{2}$. Indeed,\footnote{So, Dahl considers the Legendre transformation $\frac{1}{2}\vd^2 L$. However, with this, one can see, e.g., that the Hamiltonian corresponding to $L$ would be $0$. With the usual convention, which is Minguzzi's and ours, the Hamiltonian corresponding to a $2$-hom. Lagrangian is itself but defined on $\TT^*M$. By contrast, our $\varphi$ is  Minguzzi's $\frac{1}{2}g$ and Dahl's $g$, the latter being the most usual convention.} denoting 
\begin{equation}
	\varphi:=\frac{1}{2}\vd^2 L,
	\label{varphi}
\end{equation}
  for any $v\in A_p\subset A$ and by applying \eqref{alg_euler} to $\vd L$ (so $\alpha=1$), 
\[
\left(\vd L\right)_v(-)=\left(\imath_{\Can}\vd\vd L\right)_v(-)=\left(\vd^2L\right)_v(-,\Can_v)=2\varphi_v(-,v)\in\TT^\ast_{p}M.
\]

Now, way beyond semi-Finsler geometry, the concept of Legendre transformation is central to e.g. Lagrangian mechanics. It could be defined for any regular Lagrangian, providing maps from (a subset of) $\TT_pM$ to $\TT^\ast_{p}M$ \cite[\S 4.2]{HSS}. We will explore this concept accordingly with the level above that of semi-Finsler Lagrangians, i.e. $\mathrm{h}_1\Ten^0_1\rightarrow\mathrm{h}_2\Ten^0_0$ in \eqref{ladder}, but maintaining a nondegeneracy condition. (This is needed e.g. to locally map a Lagrangian on $\TT M$ to a Hamiltonian on $\TT^\ast M$.)
\begin{definition} \label{Legendre}
	An anisotropic $1$-form $\ell\in\Ten_1^0(M_A)$ will be called a \emph{Legendre transformation} provided that for each $p\in M$, the map $v\in A_p\mapsto\ell_v=\ell_i(v)\,\dd x^i_{p}\in\TT^\ast_{p}M$ is a local diffeomorphism onto its image. Equiv., provided that the anisotropic tensor $\vd\ell=\ell_{\cdot j}\,\dd x^i\otimes\dd x^j$ is nondegenerate on all of $A$. We shall always assume that our Legendre transformations are $1$-homogeneous, so $\ell\in\mathrm{h}_1\Ten_1^0(M_A)$. The set of all Legendre transformations will be denoted by $\mathscr{M}_{\mathrm{Lt}}(A)$.
\end{definition}
It is trivially true that the Legendre transformation associated with a semi-Finsler Lagrangian $L$ indeed satisfies Def. \ref{Legendre}. Thus, the injective $\underset{2}{\vd}\colon\mathrm{h}_2\Ten^0_0\rightarrow\mathrm{h}_1\Ten^0_1$ actually maps $\mathscr{M}_{\mathrm{s-F}}(A)$ to $\mathscr{M}_{\mathrm{Lt}}(A)$. As per Th. \ref{main th} for $(\alpha,\beta)=(1,2)$ (recall here $(r,\omega)=(0,2)$), each  $\ell\in\mathscr{M}_{\mathrm{Lt}}(A)$ corresponds to a certain pair $(\underset{0}{\ell},\underset{1}{\Delta})\in\mathrm{h}_{2}\Ten^0_{0}\times\Ker(\overset{2}{\imath_{\Can}})$ with $\underset{0}{\ell}=\frac{1}{2}\,\overset{2}{\imath_{\Can}}\,\ell$. But, for instance, it is not clear whether $\underset{0}{\ell}\in\mathscr{M}_{\mathrm{s-F}}(A)$.
 
When considering anisotropic extensions of relativity (that is, violating Lorentz symmetry), there is an obvious notion more general than a fundamental tensor \eqref{fundamental tensor}. Namely, that of a Lorentzian scalar product that differs for each \emph{observer} $v\in A_p$. Such a collection $g=\left(g_v\right)$ of scalar products arises in various settings:
\begin{itemize}
	\item It is one of the dynamical variables of García-Parrado and Minguzzi's recent theory \cite{GPMin}. (See \cite{Min23} for work that relates it with Lorentz-Finsler relativity.)
	
	\item Anisotropic gravitational theories from other authors, such as Vacaru \cite{Vac}, also make full sense for non-Finslerian $g$'s.
	
	\item It is the subject of Miron's generalized Lagrange geometry \cite{Mir} (see also \cite[Ch. X]{MiAn}), yielding applications to e.g. relativistic optics \cite{MK}.
 
\end{itemize} 
\begin{definition}\label{anis_met}
	A symmetric anisotropic tensor $g\in\Ten_2^0(M_A)$ will be called an \emph{anisotropic metric} provided that it is nondegenerate on all of $A$. We shall always assume that our anisotropic metrics are $0$-homogeneous, so $g\in\mathrm{h}_0\Ten_2^0(M_A)$. We denote the set of anisotropic metrics by $\mathscr{M}_{\mathrm{anis}}(A)$.
\end{definition}

The injection $\underset{1}{\vd}\colon\mathrm{h}_1\Ten^0_1\rightarrow\mathrm{h}_0\Ten^0_2$ does not quite map Legendre transformations $\ell$ to anisotropic metrics, for $\vd\ell$ does not need to be symmetric (and its symmetrization could degenerate). Applying Th. \ref{main th} with $(\alpha,\beta)=(0,1)$ and  $(\alpha,\beta)=(0,2)$ shows that each $g\in \mathscr{M}_{\mathrm{anis}}(A)$ is equivalent to certain
\begin{equation}
	\begin{matrix}
		(\underset{1}{g},\underset{2}{\Delta})&\in&\mathrm{h}_1\Ten^0_1\times\Ker(\overset{1}{\imath_{\Can}}) \\
		\rotatebox[origin=c]{90}{$\equiv$}& & \rotatebox[origin=c]{90}{$\equiv$}\\
		(\underset{0}{g},\underset{1}{\Delta},\underset{2}{\Delta}) & \in&\mathrm{h}_2\Fun\times\Ker(\overset{2}{\imath_{\Can}})\times\Ker(\overset{1}{\imath_{\Can}})
	\end{matrix}
\label{decomposition anis metric}
\end{equation}
with $\underset{1}{g}=\overset{1}{\imath_{\Can}}\,g$, $\underset{0}{g}=\frac{1}{2}\overset{2}{\imath_{\Can}}\overset{1}{\imath_{\Can}}\,g$. But, as above, it is not clear what nondegeneracy conditions $\underset{1}{g}$ or $\underset{0}{g}$ will fulfill.

\subsection{The ladder vs. symmetry} \label{symmetry}

In \eqref{decomposition anis metric}, it does not appear easy to characterize the symmetry of $g=\vd\underset{1}{g}+\underset{2}{\Delta}=\vd\vd\underset{0}{g}+\vd\underset{1}{\Delta}+\underset{2}{\Delta}\in \mathrm{h}_0\Ten_2^0(M_A)$ in terms of $\underset{1}{\Delta}$ and $\underset{2}{\Delta}$. This suggests the general problem, in Th. \ref{main th}, of determining which elements of $\mathrm{h}_{\beta}\Ten^r_{\omega-\beta}\times\Ker(\overset{\beta}{\imath_{\Can}})\times\ldots\times\Ker(\overset{\alpha+2}{\imath_{\Can}})\times\Ker(\overset{\alpha+1}{\imath_{\Can}})$ correspond to those in $\mathrm{h}_{\alpha}\Ten^r_{\omega-\alpha}$ that are symmetric in their covariant indices. However, it is to be expected that such a problem does not have an easy solution. Indeed, as mentioned in Rem. \ref{synthesis}, when changing conventions for $\imath_{\Can}$ and $\vd$, the residues $(\underset{\omega-\beta+1}{\Delta},\ldots,\underset{\omega-\alpha-1}{\Delta},\underset{\omega-\alpha}{\Delta})$ would change and only the symmetric $\vd^{\omega-\alpha}\underset{0}{S}$ would remain invariant. Therefore, it seems reasonable that the $\underset{s}{\Delta}$ are not so suitable to characterize the symmetry of $\underset{\omega-\alpha}{S}$. We will not attempt to adapt the theory to deal with the subspaces of \eqref{ladder} that are symmetric it their covariant entries. Because of this, we will usually have to \emph{assume} that a given $g\in\mathrm{h}_{0}\Ten^0_{2}$ is symmetric in order to have it be an anisotropic metric.

The situation is different if one wants to characterize the symmetry of the $\emph{vertical derivative}$ of an $\ell\in\mathrm{h}_{1}\Ten^0_{1}$. If $\ell=\vd\underset{0}{\ell}$ for an $\underset{0}{\ell}\in\mathrm{h}_{2}\Ten^0_{0}$, then obviously $\vd\ell$ is symmetric, and this sufficient condition is also necessary.
\begin{proposition} \label{characterization}
	For a given $\ell\in\mathscr{M}_{\mathrm{Lt}}(A)$, the following are equivalent: 
	\begin{enumerate}[label=\it{\roman*)}]
		\item \label{characterization 1} Its vertical derivative is an anisotropic metric, i.e.,  $\vd\ell\in\mathscr{M}_{\mathrm{anis}}(A)$.
		
		\item \label{characterization 2} $\underset{1}{\Delta}=0$.
		
		\item \label{characterization 3} $\ell$ is the transformation associated with a (unique) Lagrangian $L\in\mathscr{M}_{\mathrm{s-F}}(A)$.
	\end{enumerate}
\end{proposition} 
\begin{proof}
	The general case is $\ell=\vd\underset{0}{\ell}+\underset{1}{\Delta}$, where the residue with respect to $\mathrm{h}_{2}\Ten^0_{0}$ is 
	\[
	\underset{1}{\Delta_i}=\ell_i-\left(\frac{1}{2}\ell_ay^a\right)_{\cdot i}=\ell_i-\frac{1}{2}\ell_{a\,\cdot i}y^a-\frac{1}{2}\ell_i=\frac{1}{2}\ell_i-\frac{1}{2}\ell_{a\,\cdot i}y^a
	\]
	(recall \eqref{proj ker}). This way, if $\vd\ell$ is symmetric, then
	\[
	\underset{1}{\Delta_i}=\frac{1}{2}\ell_i-\frac{1}{2}\ell_{i\,\cdot a}y^a=\frac{1}{2}\ell_i-\frac{1}{2}\ell_i=0.
	\]
	In this latter case, of course, $\ell=\vd\underset{0}{\ell}$. We have just proven \ref{characterization 1} $\implies$ \ref{characterization 2} $\implies$ \ref{characterization 3}, whereas \ref{characterization 3} $\implies$ \ref{characterization 1} was commented above.
\end{proof} 

\begin{remark} \label{characterization remark}
	Analogously, consider the well-known characterization that a $g\in\mathscr{M}_{\mathrm{anis}}(A)$ is the fundamental tensor of a semi-Finsler Lagrangian if and only if $\vd g$ is totally symmetric \cite[Th. 3.4.2.1]{AIM}. One can also deduce this by relating the components of the $\underset{1}{\Delta}$ and $\underset{2}{\Delta}$ of \eqref{decomposition anis metric} with $g_{ij\,\cdot k}-g_{ik\,\cdot j}$. This and Prop. \ref{characterization} are instances of a general relation of the residues in Th. \ref{main th} with symmetry defects of vertical derivatives of $\underset{\omega-\alpha}{S}$. Such a relation is straightforward but lengthy to obtain, so we shall not write it down explicitly.
\end{remark}

\subsection{The ladder vs. signature} \label{regularity}

For $g\in\mathrm{h}_0\Ten^0_2$, if 
\[
\overset{2}{\imath_{\Can}}\overset{1}{\imath_{\Can}}\,g=L
\] 
and $g$ is an anisotropic metric, we already mentioned that it is possible for $L$ not to be a semi-Finsler Lagrangian. Denoting, as in \eqref{varphi}, $	\varphi:=\frac{1}{2}\vd^2L\in\mathrm{h}_0\Ten^0_2$, it is also possible that $g$ is denegerate while $\varphi$ is not. But, for the relativistic applications, it is important to point out that even if $g,\varphi\in\mathscr{M}_{\mathrm{anis}}(A)$, their signatures may be different. Let us see a general example of this phenomenon, where $\varphi$ is obtained by performing a sort of observer-dependent Wick rotation.
\begin{example} \label{example regularity}
	Now, let $L\in\mathrm{h}_2\Fun(A)$ be given with $\varphi$ nondegenerate, assume that  $L$ never vanishes on $A$ and take a parameter $\kappa\in\mathbb{R}$. We use this to define a symmetric $g\in\mathrm{h}_0\Ten^0_2$ whose projection to the ground floor $\mathrm{h}_2\Ten^0_0$ will be essentially $L$. It is given, for any  $v\in A_p\subset A$, by
	\begin{equation}
		g_v(u,w):=\varphi_v(u,w)+\kappa\frac{\varphi_v(v,u)\,\varphi_v(v,w)}{L(v)}\quad\left(u,w\in\TT_p M\right).
		\label{g example}
	\end{equation}
	Some observations on $g$:
	\begin{enumerate}
		\item By Euler's theorem, $\varphi_v(v,v)=L(v)$ and then
		\begin{equation}
			g_v(v,-)=\varphi_v(v,-)+\kappa\frac{\varphi_v(v,v)\,\varphi_v(v,-)}{L(v)}=\left(1+\kappa\right)\varphi_v(v,-).
			\label{example auxiliar}
		\end{equation}
		\item Let $E_v:=\Ker(\varphi_v(v,-))\subset\TT_pM$; since $\varphi_v(v,v)=L(v)\neq 0$, it is a hyperplane transverse to $v$.  %be defined as the orthogonal complement to $\left\{v\right\}$ according to $\varphi_v$. 
		From  \eqref{g example}, it is clear that 
		\[
		\left.g_v\right|_{E_v\times E_v}=\left.\varphi_v\right|_{E_v\times E_v}.
		\]
		Thus, the set composed of the orthogonal bases $\left\{e_1,\ldots,e_{n-1}\right\}$ of $E_v$ is the same when computed with respect to $g_v$ and to $\varphi_v$. 
		
		\item This allows for the comparison of the signatures of $g_v$ and $\varphi_v$. Indeed, take the basis  $\mathcal{B}:=\left\{v,e_1,\ldots,e_{n-1}\right\}$ of $\TT_pM$: by construction, it is $\varphi_v$-orthogonal, and by \eqref{example auxiliar}, it is also $g_v$-orthogonal. Putting
		\[
		\mathrm{Mat}(\varphi_v,\mathcal{B})=\mathrm{diag}(L(v),\epsilon_1,\ldots,\epsilon_n),
		\]
		we obtain that 
		\[
		\mathrm{Mat}(g_v,\mathcal{B})=\mathrm{diag}(\left(1+\kappa\right)L(v),\epsilon_1,\ldots,\epsilon_n).
		\]
		As a result:
		\begin{itemize}
			\item If $\kappa>-1$, then $g_v$ and $\varphi_v$ share their signature.
			
			\item If $\kappa=-1$, then $g_v$ is degenerate (though $\varphi_v$ was not).
			
			\item If $\kappa<-1$, then the signature changes from $g_v$ to $\varphi_v$, switching a negative sign to a positive one in case that $L(v)>0$ and the opposite in case that $L(v)<0$. This allows, among others, for transitions between positive definite metrics and Lorentzian ones (taking these with signature $(-,+,\ldots,+)$). 
		\end{itemize}
		
		\item The result of destroying the residues of $g$ (as in \eqref{destroying defects}) is $\varphi$ up to a constant factor. Indeed, recalling \eqref{sequence projections}, as a first step we will apply $\frac{1}{2}\,\overset{2}{\imath_{\Can}}\circ\overset{1}{\imath_{\Can}}$ to project $g$ down to $\mathrm{h}_2\Fun$. By \eqref{example auxiliar}, as announced,
		\[
		\left(\frac{1}{2}\imath_{\Can}\imath_{\Can}\,g\right)(v)=\left(\frac{1}{2}g(\Can,\Can)\right)(v)=\frac{1}{2}g_v(v,v)=\frac{1+\kappa}{2}\varphi_v(v,v)=\frac{1+\kappa}{2}L(v).
		\]
		As a second step, we apply $\underset{1}{\vd}\circ\underset{2}{\vd}$ to go back to $\mathrm{h}_0\Ten^0_2$:
		\[
		\frac{1+\kappa}{2}\left(\vd\vd L\right)_v(-,-)=\left(1+\kappa\right)\varphi_v(-,-).
		\]
	\end{enumerate}
\end{example}

\section{Applications to connection-type objects} \label{section connections}

\subsection{From sprays to anisotropic connections}

The plethora of connections, generalizing the affine ones, that are used in semi-Finsler geometry is well known, see e.g. \cite{KLS}. There have been works establishing relations between the various kinds of objects that generalize the classical connections (\emph{connection-type objects}, from now on) \cite{Dahl,Min14,JSV1}. Still, in these, some of the possible relations might be missing or, at any rate, the different kinds are not treated on an equal formal footing. (For instance, one usually does not think of an spray as a type of connection, but our approach will make apparent that it is always consistent to do so.) We believe that the ladder structure studied in \S \ref{section ladder} makes accessible \emph{all} the interrelations between the classes of interest here. In order to transport the ladder to connection-type objects in place of tensors, we shall first define the affine bundles over $A$ of which the former are sections. Later, we will give a statement that provides the ladder transitions between the levels of sprays, nonlinear connections and anisotropic ones. In particular, we will end up expressing the results of \cite{JSV2} in a more synthetic way and complementing them.

The \emph{symmetrized bundle} \cite[\S 2.4]{Min14} (see also \cite[Prop. 4.1.3]{BM}) is $\mathrm{S}A:=\left\{\xi\in\TT_vA\colon\; v\in A,\;\dd(\pi_A)_v(\xi)=v\right\}$. The \emph{$1$-jet bundle} \cite[\S 12.16]{KMS} can be introduced as\footnote{This identification is so that if $V\colon U\subseteq M\rightarrow A$ is a local section, its $1$-jet prolongation is simply given by $p\in U\mapsto\dd V_p\in\mathrm{J}^1_{V(p)}A\subset\mathrm{Hom}(\TT_pM,\TT_{V(p)}A)$.} $\mathrm{J}^1A=\left\{\eta\in\mathrm{Hom}(\TT_{\pi(v)}M,\TT_v A)\colon\; v\in A,\;\dd(\pi_A)_v\circ\eta=\mathrm{Id}\right\}$. Besides, one can easily construct the \emph{connection bundle} $\mathrm{C}M$, whose sections are the affine connections on $M$. For instance, one can declare two such affine connections $\nabla^1$ and $\nabla^2$ to be \emph{equivalent at $p\in M$} if their Christoffel symbols coincide at $p$, writing then $\left[\nabla^1\right]_{p}=\left[\nabla^2\right]_{p}$. With this, 
\begin{equation}
	\mathrm{C}_pM=\left\{\left[\nabla\right]_{p}\colon \nabla\text{ affine connection on $M$}\right\},\qquad \mathrm{C}M=\underset{p\in M}{\bigcup} \mathrm{C}_pM,
	\label{connection bundle}
\end{equation}
 the latter with the appropriate fiber bundle structure over $M$. (In \cite[\S 2.2]{JSV2} we introduced $\mathrm{C}M$ by directly specifying its transformation cocycle.)

One easily sees that $\mathrm{S}A$ is an affine subbundle of $\mathrm{T}A\rightarrow A$ directed by $\VV A=\left\{\chi\in\TT_vA\colon\; v\in A,\;\dd(\pi_A)_v(\chi)=0\right\}$ or, equiv., by $\pi_A^*(\TT M)$ (recall \eqref{vertical isomorphism}). In the same vein, $\mathrm{J}^1A$ is an affine subbundle of $\mathrm{Hom}(\pi_A^*(\TT M),\TT A)$ directed by $\mathrm{Hom}(\pi_A^*(\TT M),\VV A)\equiv\pi_A^*(\TT^\ast M)\otimes\VV A\equiv\pi_A^*(\TT^\ast M\otimes\TT M)$. As for $\mathrm{C}M$, it is an affine bundle directed by $\TT M\otimes\TT^\ast M\otimes\TT^\ast M\rightarrow M$, so the bundle that we are really interested in, which is $\pi_A^*(\mathrm{C}M)$, is affine directed by $\pi_A^*(\TT M\otimes\TT^\ast M\otimes\TT^\ast)\rightarrow A$. One can compare the following definitions with \cite[Def. 4.1.1]{Shen}, \cite[Def. 5.1.18]{KLS}, \cite[\S 1.1]{AP}, \cite[Def. 2.6]{JSV2}, \cite[Def. 4]{JSV1} and \cite[Def. 2.10]{Jav19} among others.
\begin{definition}
	\,
	\begin{enumerate}
		\item A \emph{spray} is a section $G\colon A\rightarrow\mathrm{S}A$ that is \emph{$2$-homogeneous}, in that if $v\mapsto G_v$, then $\lambda v\mapsto\lambda\,\dd(h_\lambda)_v(G_v)$ for any $\lambda\in\mathbb{R}^+$. We will denote the set of all sprays by $\mathscr{C}_{\mathrm{spr}}(A)$.
		
		\item A \emph{nonlinear connection} is a section $N\colon A\rightarrow\mathrm{J}^1A$. We will always take our nonlinear connections to be \emph{$1$-hom.}, in that if $v\mapsto N_v$, then $\lambda v\mapsto\dd(h_{\lambda})_v\circ N_v$, and we will denote the set of all of them by $\mathscr{C}_{\mathrm{nl}}(A)$. 
		
		\item An \emph{anisotropic connection} is a section $\Gamma\colon A\rightarrow\pi_A^*(\mathrm{C}M)$. We will take these to be \emph{$0$-hom.}, in the sense that if $v\mapsto \Gamma_v\in\mathrm{C}_{\pi(v)}M$, then $\lambda v\mapsto\Gamma_v\left(\in\mathrm{C}_{\pi(\lambda v)}M=\mathrm{C}_{\pi(v)}M\right)$ too, and we will denote the set of all of them by $\mathscr{C}_{\mathrm{anis}}(A)$.
	\end{enumerate}
\end{definition}
Let us write down the local expressions of these objects in natural coordinates on each of the bundles $\mathrm{S}A$, $\mathrm{J}^1A$ and $\pi_A^*(\mathrm{C}M)$. This will clarify the affine bundle structures, the fact that we can recover the objects in their usual appearances (e.g., covariant derivative operators) and how to transport the ladder structure from \S \ref{section ladder}. Given natural coordinates $\left(x^i,\y^i\right)$
and\footnote{\label{footnote sprays} We choose to label the component of $\xi$ on $\left.\partial_{y^i}\right|_v$ as $-2\,\xi^i$ instead of $\xi^i$ in order to maintain the usual convention for sprays. These are usually denoted as $y^i\partial_{x^i}-2\,G^i\partial_{y^i}$. Had we chosen to write them as $y^i\partial_{x^i}`+G^i\partial_{y^i}$, we would have had to change our way of introducing the operator $\vd$ on them accordingly.} 
\[
\xi=\overset{0}{\xi^i}\left.\partial_{x^i}\right|_v-2\,\xi^i\left.\partial_{y^i}\right|_v\in\mathrm{S}_vA\subset\mathrm{S}A,
\]
its defining condition $\dd(\pi_A)_v(\xi)=v$ translates into $\overset{0}{\xi^i}=y^i(v)$. Therefore, a spray $G\in\mathscr{C}_{\mathrm{spr}}(A)$ is expressed as 
\[
G_v=y^i(v)\left.\partial_{x^i}\right|_v-2\,G^i(v)\left.\partial_{y^i}\right|_v\in\mathrm{S}_vA;\qquad G^i(\lambda v)=\lambda^2 G^i(v)
\]
for any $\lambda\in\mathbb{R}^+$. Under a change of chart $\left(x^i,\y^i\right)\rightsquigarrow\left(\wt{x}^i,\wt{y}^i\right)$, one can use the transformation laws of $\partial_{x^i}$ and $\partial_{y^i}$ to find that the corresponding components are 
\begin{equation}
	\wt{G}^i=-\frac{1}{2}\frac{\partial^2\wt{x}^i}{\partial x^b\partial x^c}y^by^c+\frac{\partial\wt{x}^i}{\partial x^a}G^a.
	\label{transformation spray}
\end{equation}
Now, given\footnote{Analogous comments to those of footnote \ref{footnote sprays}. Now, they are in order to maintain the convention of writing the covariant derivative of a local section $V\colon U\rightarrow A$ with respect to $N\in\mathscr{C}_{\mathrm{nl}}(A)$ as $\mathrm{D}V=\left(\frac{\partial V^i}{\partial x^j}+N^i_j(V)\right)\partial_{x^i}\otimes\dd x^j$.}
\[
\eta=\left(\overset{0}{\eta^i_j}\left.\partial_{x^i}\right|_v-\eta^i_j\left.\partial_{y^i}\right|_v\right)\otimes\dd x^j_{\pi(v)}\in\mathrm{J}^1_vA\subset\mathrm{J}^1A,
\]
the condition $\dd(\pi_A)_v\circ\eta=\mathrm{Id}$ translates into $\overset{0}{\eta^i_j}=\delta^i_j$, so a nonlinear connection $N\in\mathscr{C}_{\mathrm{nl}}(A)$ is expressed as
\[
N_v=\left(\delta^i_j\left.\partial_{x^i}\right|_v-N^i_j(v)\left.\partial_{y^i}\right|_v\right)\otimes\dd x^j_{\pi(v)}\in\mathrm{J}^1_vA;\qquad N^i_j(\lambda v)=\lambda N^i_j(v).
\]
This time, under changes $\left(x^i,\y^i\right)\rightsquigarrow\left(\wt{x}^i,\wt{y}^i\right)$, the transformation laws of $\partial_{x^i}$,  $\partial_{y^i}$ and $\dd x^j$ yield
\begin{equation}
	\wt{N}^i_j=-\frac{\partial^2\wt{x}^i}{\partial x^b\partial x^c}\frac{\partial x^b}{\partial \wt{x}^j}y^c+\frac{\partial\wt{x}^i}{\partial x^a}\frac{\partial x^b}{\partial\wt{x}^j}N^a_b.
	\label{transformation nonlinear}
\end{equation} 
Last (recall \eqref{connection bundle}), an anisotropic connection $\Gamma\in\mathscr{C}_{\mathrm{anis}}(A)$ is expressed as 
\[
\Gamma_v=\left[\nabla\right]_{\pi(v)}\in\mathrm{C}_{\pi(v)}M,\quad\left[\nabla\right]_{\pi(v)}\equiv\left(\Gamma^i_{jk}(v)\right);\qquad\Gamma^i_{jk}(v)=\Gamma^i_{jk}(\lambda v),
\]
and the well-known transformation law of the classical Christoffel symbols gives
\begin{equation}
	\wt{\Gamma}^i_{jk}(v)=-\frac{\partial^2\wt{x}^i}{\partial x^b\partial x^c}\frac{\partial x^b}{\partial\wt{x}^j}\frac{\partial x^c}{\partial\wt{x}^k}+\frac{\partial\wt{x}^i}{\partial x^a}\frac{\partial x^b}{\partial\wt{x}^j}\frac{\partial x^c}{\partial\wt{x}^k}\Gamma^a_{bc}.
	\label{transformation anisotropic}
\end{equation}

In the next subsection, we extend the operators $\imath_{\Can}$ and $\vd$ to act between the sets $\mathscr{C}_{\mathrm{spr}}(A)$, $\mathscr{C}_{\mathrm{nl}}(A)$ and $\mathscr{C}_{\mathrm{anis}}(A)$, eventually including $\underset{\mathrm{anis}}{\vd}\colon\mathscr{C}_{\mathrm{anis}}(A)\rightarrow\mathrm{h}_{-1}\Ten^1_3(M_A)$. (In Prop. \ref{main prop}, this operator would correspond to $\underset{0}{\vd}$ and, therefore, would not appear on subsequent the ladder.) The idea will be to locally consider distinguished elements of these affine spaces, thus establishing identifications with their directing vector spaces. Such identifications will allow one to transform 
\[
\xymatrix{ 
	\mathrm{h}_{0}\Ten^1_2(M_A)\ar@/^/[r]^{\overset{1}{\imath_{\Can}}} &  \mathrm{h}_{1}\Ten^1_1(M_A)\ar@/^/[l]^{\underset{1}{\vd}}\ar@/^/[r]^{\frac{1}{2}\overset{2}{\imath_{\Can}}} & \mathrm{h}_{2}\Ten^1_0(M_A)\ar@/^/[l]^{\underset{2}{\vd}}}
\]
(essentially the ladder of Def. \ref{def ladder} labeled by $(r,\omega)=(1,2)$) into 
\begin{equation}
	\xymatrix{ 
		\mathscr{C}_{\mathrm{anis}}(A)\ar@/^/[r]^{\overset{\mathrm{nl}}{\imath_{\Can}}} &  \mathscr{C}_{\mathrm{nl}}(A)\ar@/^/[l]^{\underset{\mathrm{nl}}{\vd}}\ar@/^/[r]^{\overset{\mathrm{spr}}{\imath_{\Can}}} & \mathscr{C}_{\mathrm{spr}}(A).\ar@/^/[l]^{\underset{\mathrm{spr}}{\vd}}}
	\label{ladder connections}
\end{equation}
The distinguished representatives will be induced by each chart $(U,x)$ on $M$. This  highlights the importance of the cocycles \eqref{transformation spray}, \eqref{transformation nonlinear} and \eqref{transformation anisotropic} when put together: their compatibility with $\imath_{\Can}$ (contracting with $y^i$) and $\vd$ (applying $_{\cdot i}$) is totally transparent. 

\subsection{The ladder of connection-type objects}
 
To be precise, let $G^{(U,x)}$, $N^{(U,x)}$ and $\Gamma^{(U,x)}$ denote, resp., the spray, nonlinear connection and anisotropic connection defined on $\left.A\right|_U=A\cap\TT U$ whose components in the natural chart $(\TT U,(x,y))$ are $0$:
\[
G^{(U,x)}=y^i\,\partial_{x^i},\qquad N^{(U,x)}=\delta^i_j\,\partial_{x^i}\otimes\dd x^j,\qquad\Gamma^{(U,x)}\equiv\left\{\Gamma^i_{jk}=0\right\}. 
\]
Keep in mind that, due to our conventions on the different components, when $G\in\mathscr{C}_{\mathrm{spr}}(A)$, $N\in\mathscr{C}_{\mathrm{nl}}(A)$, $\Gamma\in\mathscr{C}_{\mathrm{anis}}(A)$ and $Z=Z^i\,\partial_{x^i}\in\mathrm{h}_2\Ten^1_0(M_A)$, $J=J^i_j\,\partial_{x^i}\otimes\dd x^j\in\mathrm{h}_1\Ten^1_1(M_A)$, $P=P^i_{jk}\,\partial_{x^i}\otimes\dd x^j\otimes\dd x^k\in\mathrm{h}_0\Ten^1_2(M_A)$, we have
\[
G-2Z=y^i\,\partial_{x^i}-2\left(G^i+Z^i\right)\in\mathscr{C}_{\mathrm{spr}}(A),
\]
\[
N-J=\left\{\delta^i_j\,\partial_{x^i}-\left(N^i_j+J^i_j\right)\partial_{y^i}\right\}\otimes\dd x^j\in\mathscr{C}_{\mathrm{nl}}(A),
\]

\[
\Gamma+P\equiv\left\{\Gamma^i_{jk}+P^i_{jk}\right\}\in\mathscr{C}_{\mathrm{anis}}(A).
\]

\begin{proposition} \label{prop connections}
	Let $(U,x)\rightsquigarrow(\wt{U},\wt{x})$ be a change of chart on $M$ and $G\in\mathscr{C}_{\mathrm{spr}}(\left.A\right|_{U\cap\wt{U}})$, $N\in\mathscr{C}_{\mathrm{nl}}(\left.A\right|_{U\cap\wt{U}})$, $\Gamma\in\mathscr{C}_{\mathrm{anis}}(\left.A\right|_{U\cap\wt{U}})$. 
	\begin{enumerate} [label=\it{\roman*)}]
		\item Putting 
		\[
		G=G^{(U,x)}-2Z^{(U,x)}=G^{(\wt{U},\wt{x})}-2Z^{(\wt{U},\wt{x})},
		\]
		\begin{equation}
			N=N^{(U,x)}-J^{(U,x)}=N^{(\wt{U},\wt{x})}-J^{(\wt{U},\wt{x})},
			\label{coherence nonlinear 0}
		\end{equation}
		\[
		\Gamma=\Gamma^{(U,x)}+P^{(U,x)}=\Gamma^{(\wt{U},\wt{x})}+P^{(\wt{U},\wt{x})}
		\]
		for unique $Z^{(U,x)},Z^{(\wt{U},\wt{x})}\in\mathrm{h}_2\Ten^1_0((U\cap\wt{U})_A)$, $J^{(U,x)},J^{(\wt{U},\wt{x})}\in\mathrm{h}_1\Ten^1_1((U\cap\wt{U})_A)$, and $P^{(U,x)},P^{(\wt{U},\wt{x})}\in\mathrm{h}_0\Ten^1_2((U\cap\wt{U})_A)$,
		one has that 
		\begin{equation}
			N^{(U,x)}-\vd Z^{(U,x)}=N^{(\wt{U},\wt{x})}-\vd Z^{(\wt{U},\wt{x})},
			\label{coherence spray}
		\end{equation}
		\begin{equation}
			G^{(U,x)}-\imath_{\Can}J^{(U,x)}=G^{(\wt{U},\wt{x})}-\imath_{\Can}J^{(\wt{U},\wt{x})},\quad\Gamma^{(U,x)}+\vd J^{(U,x)}= \Gamma^{(\wt{U},\wt{x})}+\vd J^{(\wt{U},\wt{x})},
			\label{coherence nonlinear}
		\end{equation}
	
		\begin{equation}
			N^{(U,x)}-\imath_{\Can} P^{(U,x)}=N^{(\wt{U},\wt{x})}-\imath_{\Can} P^{(\wt{U},\wt{x})},\quad \vd P^{(U,x)} = \vd P^{(\wt{U},\wt{x})}.
			\label{coherence anisotropic}
		\end{equation}

		\item \label{prop connections 2} Consequently, there appear well-defined maps $\underset{\mathrm{spr}}{\vd}\colon\mathscr{C}_{\mathrm{spr}}(A)\rightarrow \mathscr{C}_{\mathrm{nl}}(A)$, $\overset{\mathrm{spr}}{\imath_{\Can}}\colon\mathscr{C}_{\mathrm{nl}}(A)\rightarrow\mathscr{C}_{\mathrm{spr}}(A)$, $\underset{\mathrm{nl}}{\vd}\colon\mathscr{C}_{\mathrm{nl}}(A)\rightarrow \mathscr{C}_{\mathrm{anis}}(A)$, $\overset{\mathrm{nl}}{\imath_{\Can}}\colon\mathscr{C}_{\mathrm{anis}}(A)\rightarrow\mathscr{C}_{\mathrm{nl}}(A)$, $\underset{\mathrm{spr}}{\vd}\colon \mathscr{C}_{\mathrm{anis}}(A)\rightarrow\mathrm{h}_{-1}\Ten^1_3(M_A)$. For instance, for $G\in\mathscr{C}_{\mathrm{spr}}(A)$, one locally expresses it as $G^{(U,x)}-2Z^{(U,x)}$ and then, on $\left.A\right|_U$, puts $\underset{\mathrm{spr}}{\vd} G:=N^{(U,x)}-\vd Z^{(U,x)}$. The other maps are defined analogously.
	\end{enumerate}
\end{proposition}

\begin{proof}
	\!
	\begin{enumerate} [label=\it{\roman*)}]
		\item Let us start with \eqref{coherence nonlinear}. First, we will express \eqref{coherence nonlinear 0} in the $\wt{x}$ coordinates. Concretely, with the transformation law \eqref{transformation nonlinear} and the defining property of $N^{(U,x)}$ and $N^{(\wt{U},\wt{x})}$ (i.e., $\left(N^{(U,x)}\right)^i_j=0$ and $\wt{\left(N^{(\wt{U},\wt{x})}\right)}^i_j=0$):
		\[
		J^{(U,x)}=\left(N^{(U,x)}-N^{(\wt{U},\wt{x})}\right)+J^{(\wt{U},\wt{x})}=\left\{\left(\frac{\partial^2\wt{x}^i}{\partial x^b\partial x^c}\frac{\partial x^b}{\partial \wt{x}^j}y^c-0\right)+\wt{\left(J^{(\wt{U},\wt{x})}\right)}^i_j\right\}\partial_{\wt{x}^i}\otimes\dd\wt{x}^j.
		\]
		Then, using this and also \eqref{transformation spray} for $G^{(U,x)}$,
		\[
		\begin{split}
			G^{(U,x)}-\imath_{\Can}J^{(U,x)}&=y^i\,\partial_{x^i}-\left\{2\cdot 0+\left(J^{(U,x)}\right)^i_ay^a\right\}\partial_{y^i}\\ &=\wt{y}^i\,\partial_{\wt{x}^i}-\left\{2\wt{\left(G^{(U,x)}\right)}^i+\wt{\left(J^{(U,x)}\right)}^i_a\wt{y}^a\right\}\partial_{\wt{y}^i}\\
			&= 
			\wt{y}^i\,\partial_{\wt{x}^i}-\left\{-\frac{\partial^2\wt{x}^i}{\partial x^b\partial x^c}y^by^c+\frac{\partial^2\wt{x}^i}{\partial x^b\partial x^c}\frac{\partial x^b}{\partial \wt{x}^a}y^c\wt{y}^a+\wt{\left(J^{(\wt{U},\wt{x})}\right)}^i_a\wt{y}^a\right\}\partial_{\wt{y}^i} \\
			&=\wt{y}^i\,\partial_{\wt{x}^i}-\left\{2\cdot 0+\wt{\left(J^{(\wt{U},\wt{x})}\right)}^i_a\wt{y}^a\right\}\partial_{\wt{y}^i} \\
			&=G^{(\wt{U},\wt{x})}-\imath_{\Can}J^{(\wt{U},\wt{x})}.
		\end{split}
		\]
		On the other hand, using again the above identity and also \eqref{transformation anisotropic} for $\Gamma^{(U,x)}$,
		\[
		\begin{split}
			\Gamma^{(U,x)}+\vd J^{(U,x)}&\equiv\left\{0+\left(J^{(U,x)}\right)^i_{j\,\cdot k}\right\} \\
			&\equiv\left\{\wt{\left(\Gamma^{(U,x)}\right)}^i_{jk}+\wt{\left(J^{(U,x)}\right)}^i_{j\,\cdot k}\right\} \\
			&=\left\{-\frac{\partial^2\wt{x}^i}{\partial x^b\partial x^c}\frac{\partial x^b}{\partial\wt{x}^j}\frac{\partial x^c}{\partial\wt{x}^k}+\left(\frac{\partial^2\wt{x}^i}{\partial x^b\partial x^c}\frac{\partial x^b}{\partial \wt{x}^j}y^c\right)_{\cdot k}+\wt{\left(J^{(\wt{U},\wt{x})}\right)}^i_{j\,\cdot k}\right\} \\
			&= \left\{0+\wt{\left(J^{(\wt{U},\wt{x})}\right)}^i_{j\,\cdot k}\right\} \\
			&\equiv \Gamma^{(\wt{U},\wt{x})}+\vd J^{(\wt{U},\wt{x})}.
		\end{split}
		\]
		
		With the analogous computations, one expresses $Z^{(U,x)}$ in terms of $Z^{(\wt{U},\wt{x})}$ in the $\wt{x}$ and combines this with \eqref{transformation nonlinear} to obtain \eqref{coherence spray}. Finally, one expresses $P^{(U,x)}$ in terms of $P^{(\wt{U},\wt{x})}$ and, with \eqref{transformation anisotropic}, obtains \eqref{coherence anisotropic}.
		
		\item Note that \eqref{coherence spray} is just expressing that $\left.\underset{\mathrm{spr}}{\vd} G\right|_{A\cap\TT U}=N^{(U,x)}-\vd Z^{(U,x)}$ is independent of the chart chosen to express $G$. Consequently, all the obtained $\left.\underset{\mathrm{spr}}{\vd}G\right|_{A\cap\TT U}$ glue together smoothly to define $\underset{\mathrm{spr}}{\vd}G$ on all of $A$. In the same vein, \eqref{coherence nonlinear} establishes the well-definedness of $\overset{\mathrm{spr}}{\imath_{\Can}} N$ and $\underset{\mathrm{nl}}{\vd}N$, and \eqref{coherence anisotropic}, that of $\overset{\mathrm{nl}}{\imath_{\Can}}\Gamma$ and $\underset{\mathrm{anis}}{\vd}\Gamma$ (here $N\in\mathscr{C}_{\mathrm{nl}}(A)$ and $\Gamma\in\mathscr{C}_{\mathrm{anis}}(A)$).
	\end{enumerate}
\end{proof}

\begin{remark}
	By construction, it is clear that the operators defined in Prop. \ref{prop connections} \ref{prop connections 2} work just by contracting with $y^i$ or taking the $_{\cdot i}$ of the different components $G^i$, $N^i_j$ or $\Gamma^i_{jk}$. It is due to their coordinate expressions that we see  we are recovering the classical constructions. To be precise: 
	\begin{itemize}
		\item $\vd\colon\mathscr{C}_{\mathrm{spr}}(A)\rightarrow\mathscr{C}_{\mathrm{nl}}(A)$, $\vd G=\left(\delta^i_j\,\partial_{x^i}-G^i_{\cdot j}\,\partial_{y^i}\right)\otimes\dd x^j$, recovers, for example, \cite[Prop. 7.3.4]{KLS}, \cite[(7.9)]{Shen}, \cite[Th. 4.2.1]{BM},  \cite[(22)]{Min14}.
		
		\item $\imath_{\Can}\colon\mathscr{C}_{\mathrm{nl}}(A)\rightarrow\mathscr{C}_{\mathrm{spr}}(A)$, $\imath_{\Can}N=y^i\,\partial_{x^i}-N^i_ay^a\,\partial_{y^i}$, recovers, f. ex., \cite[Lem. and Def. 7.2.13]{KLS}, \cite[Th. 4.2.2]{BM}, \cite[(21)]{Min14}.
		
		\item  $\vd\colon\mathscr{C}_{\mathrm{nl}}(A)\rightarrow\mathscr{C}_{\mathrm{anis}}(A)$, $\vd N\equiv\left\{N^i_{j\,\cdot k}\right\}$, recovers, f. ex., \cite[Prop. and Def. 7.1.7]{KLS}, \cite[Th. 3.2.1]{BM}, \cite[\S 2-7]{Jav19}, \cite[(32)]{Min14}.
		
		\item $\imath_{\Can}\colon\mathscr{C}_{\mathrm{anis}}(A)\rightarrow\mathscr{C}_{\mathrm{nl}}(A)$, $\imath_{\Can}\Gamma=\left(\delta^i_j\,\partial_{x^i}-\Gamma^i_{ja}y^a\,\partial_{y^i}\right)\otimes\dd x^j$, recovers \cite[\S 3.1]{Jav19}.
		
		\item  $\vd\colon\mathscr{C}_{\mathrm{anis}}(A)\rightarrow\mathrm{h}_{-1}\Ten^1_2(M_A)$, $\vd \Gamma= \Gamma^i_{jk\,\cdot l}\,\partial_{x^i}\otimes\dd x^j\otimes\dd x^k\otimes\dd x^l$, recovers \cite[(35)]{Jav19}.
	\end{itemize}
(Cf. also \cite[Def. 2.8]{JSV2}, where we introduced $\overset{\mathrm{nl}}{\imath_{\Can}}$ and $\underset{\mathrm{nl}}{\vd}$.)
\end{remark}

Now that we have all the operators of the ladder \eqref{ladder connections}, let us check that they provide results analogous to Prop. \ref{main prop} and Th. \ref{main th}.
\begin{corollary} \label{corollary connections}
	One has that
	\begin{equation}
		\overset{\mathrm{spr}}{\imath_{\Can}}\circ\underset{\mathrm{spr}}{\vd}=\mathrm{Id}_{\mathscr{C}_{\mathrm{spr}}(A)},\quad\overset{\mathrm{nl}}{\imath_{\Can}}\circ\underset{\mathrm{nl}}{\vd}=\mathrm{Id}_{\mathscr{C}_{\mathrm{nl}}(A)},\quad\overset{\mathrm{0}}{\imath_{\Can}}\circ\underset{\mathrm{anis}}{\vd}=0\colon\mathscr{C}_{\mathrm{anis}}(A)\longrightarrow\mathrm{h}_0\Ten^1_2(M_A).
		\label{euler connections}
	\end{equation}
	As consequences, $\underset{\mathrm{spr}}{\vd}$ and $\underset{\mathrm{nl}}{\vd}$ are injective, $\overset{\mathrm{spr}}{\imath_{\Can}}$ and $\overset{\mathrm{nl}}{\imath_{\Can}}$ are surjective, and the following decompositions\footnote{We will also use the direct sum symbol $\oplus$ between a subspace of an affine space and one of the corresponding vector space. To be precise, in \ref{cor connnections 1} we mean that each $N\in\mathscr{C}_{\mathrm{nl}}(A)$ can be uniquely expressed as $N=\vd G-J$ for certain $G\in\mathscr{C}_{\mathrm{spr}}(A)$ and $J\in\mathrm{Ker}(\overset{2}{\imath_{\Can}})$; analogously in \ref{cor connnections 2}.} hold:
	\begin{enumerate} [label=\it{\roman*)}]
		\item \label{cor connnections 1}
		\[
		\begin{split}
			\mathscr{C}_{\mathrm{nl}}(A)&=\mathrm{Img}(\underset{\mathrm{spr}}{\vd})\oplus\mathrm{Ker}(\overset{2}{\imath_{\Can}}\colon\mathrm{h}_1\Ten^1_1(M_A)\rightarrow\mathrm{h}_2\Ten^1_0(M_A)), \\
			 N&=\underset{\mathrm{spr}}{\vd}(\overset{\mathrm{spr}}{\imath_{\Can}}N)+\left\{N-\underset{\mathrm{spr}}{\vd}(\overset{\mathrm{spr}}{\imath_{\Can}}N)\right\}.
		\end{split}
		\]
		
		\item \label{cor connnections 2}
		\[
		\begin{split}
			\mathscr{C}_{\mathrm{anis}}(A)&=\mathrm{Img}(\underset{\mathrm{nl}}{\vd})\oplus\mathrm{Ker}(\overset{1}{\imath_{\Can}}\colon\mathrm{h}_0\Ten^1_2(M_A)\rightarrow\mathrm{h}_1\Ten^1_1(M_A)), \\
			\Gamma&=\underset{\mathrm{nl}}{\vd}(\overset{\mathrm{nl}}{\imath_{\Can}}\Gamma)+\left\{\Gamma-\underset{\mathrm{nl}}{\vd}(\overset{\mathrm{nl}}{\imath_{\Can}}\Gamma)\right\}.
		\end{split}
		\]
	\end{enumerate}
\end{corollary}
\begin{proof}
	One can prove \eqref{euler connections} directly from the definitions of the operators in Prop. \ref{prop connections} \ref{prop connections 2}. For instance, given $G\in\mathscr{C}_{\mathrm{spr}}(A)$ expressed locally as $G=G^{(U,x)}-2Z^{(U,x)}$, 
	\[
	\underset{\mathrm{spr}}{\vd} G=N^{(U,x)}-\underset{2}{\vd} Z^{(U,x)},
	\]
	\[
	\overset{\mathrm{spr}}{\imath_{\Can}}(\underset{\mathrm{spr}}{\vd} G)=G^{(U,x)}-\overset{2}{\imath_{\Can}}(\underset{2}{\vd} Z^{(U,x)})=G^{(U,x)}-2Z^{(U,x)}=G
	\]
	 by applying \eqref{alg_euler} for $\alpha=2$. Analogously for the other two identities of \eqref{euler connections}. Once these are established, the injectivity and surjectivity in the statement become clear.
	 
	 In order to prove \ref{cor connnections 1} and \ref{cor connnections 2}, one first checks that $\overset{\mathrm{spr}}{\imath_{\Can}}$ and $\overset{\mathrm{nl}}{\imath_{\Can}}$ are  affine maps over the linear ones $\overset{2}{\imath_{\Can}}$ and $-\overset{1}{\imath_{\Can}}$ resp. (the latter coefficient results from our conventions). Then, writing e.g. $N\in\mathscr{C}_{\mathrm{nl}}(A)$ as 
	 $N=\underset{\mathrm{spr}}{\vd}G+J$ with $J\in\Ker(\overset{2}{\imath_{\Can}})$, one obtains that
	 \[
	 \overset{\mathrm{spr}}{\imath_{\Can}}N=\overset{\mathrm{spr}}{\imath_{\Can}}(\underset{\mathrm{spr}}{\vd}G)+\overset{2}{\imath_{\Can}}J=G,\qquad J=N-\underset{\mathrm{spr}}{\vd}G=N-\underset{\mathrm{spr}}{\vd}(\overset{\mathrm{spr}}{\imath_{\Can}}N).
	 \]
	 The fact that always $N-\underset{\mathrm{spr}}{\vd}(\overset{\mathrm{spr}}{\imath_{\Can}}N)\in\Ker(\overset{2}{\imath_{\Can}})$ results from the following computation:
	 \[
	 \overset{2}{\imath_{\Can}}\circ\left(\mathrm{Id}_{\mathscr{C}_{\mathrm{nl}}(A)}-\underset{\mathrm{spr}}{\vd}\circ\overset{\mathrm{spr}}{\imath_{\Can}}\right)=\overset{\mathrm{spr}}{\imath_{\Can}}\circ\mathrm{Id}_{\mathscr{C}_{\mathrm{nl}}(A)}-\overset{\mathrm{spr}}{\imath_{\Can}}\circ\underset{\mathrm{spr}}{\vd}\circ\overset{\mathrm{spr}}{\imath_{\Can}}=\overset{\mathrm{spr}}{\imath_{\Can}}-\overset{\mathrm{spr}}{\imath_{\Can}}=0
	 \]
	(cf. the proof of Prop. \ref{main prop}). This proves \ref{cor connnections 1}; the decomposition \ref{cor connnections 2} is analogous.
 
\end{proof}

Cor. \ref{corollary connections} allows one to write
\begin{equation}
	\mathscr{C}_{\mathrm{anis}}(A)\equiv\mathscr{C}_{\mathrm{nl}}(A)\times\mathrm{Ker}(\overset{1}{\imath_{\Can}})
	\label{referencia descomposicion}
\end{equation}
and
\[
\mathscr{C}_{\mathrm{nl}}(A)\equiv\mathscr{C}_{\mathrm{spr}}(A)\times\mathrm{Ker}(\overset{2}{\imath_{\Can}}),
\] 
obtaining the corresponding \emph{residues} $\Delta$ of any nonlinear connection and any anisotropic one. Their explicit definition would be practically as in Def. \ref{residues} and Th. \ref{main th}, so we will not repeat ourselves. Instead, let us illustrate the importance of this notion with a couple familiar examples.
\begin{example}
	The residue of $N\in\mathscr{C}_{\mathrm{nl}}(A)$ with respect to $\mathscr{C}_{\mathrm{spr}}(A)$ is precisely its \emph{torsion}, definable as the antisymmetric part of its vertical derivative: $\mathrm{Tor}^i_{jk}:=N^i_{j\,\cdot k}-N^i_{k\,\cdot j}$ Indeed, \cite[Prop. 4 (4)]{JSV1} shows that 
	\[
	\Delta^i_j=\frac{1}{2}\mathrm{Tor}^i_{ja}y^a,
	\]
	but $\Delta$ also determines $\mathrm{Tor}$:
	\[
	\mathrm{Tor}^i_{jk}=\left(\underset{0}{N_{\cdot j}^i}+\Delta^i_j\right)_{\cdot k}-\left(\underset{0}{N_{\cdot k}^i}+\Delta^i_k\right)_{\cdot j}=\underset{0}{N_{\cdot j\cdot k}^i}+\Delta^i_{j\,\cdot k}-\underset{0}{N_{\cdot k\cdot j}^i}-\Delta^i_{k\,\cdot j}=\Delta^i_{j\,\cdot k}-\Delta^i_{k\,\cdot j}.
	\]
	This is an instance of a bijective correspondence between $\mathrm{Ker}(\overset{2}{\imath_{\Can}})$ and the elements of $\mathrm{h}_0\Ten^1_2$ with certain symmetries on their vertical derivative, cf. Rem. \ref{characterization remark}.
\end{example} 
\begin{example}
	Let $\Gamma^{\mathrm{Ch}}$ denote the Chern anisotropic connection of a semi-Finsler Lagrangian $L\in\mathscr{M}_{\mathrm{s-F}}(A)$ \cite[Th. 4]{JSV1} (see also \cite[\S 2.6]{Jav19}). Then, its residue with respect to $\mathscr{C}_{\mathrm{nl}}(A)$ is the \emph{Landsberg tensor}. Indeed, this, among other ways \cite[(26)]{Jav19}, can be defined as the difference $\mathrm{Lan}:=\Gamma^{\mathrm{Ch}}-\Gamma^{\mathrm{Ber}}$, where $\Gamma^{\mathrm{Ber}}$ is the Berwald anisotropic connection, and it is well known that $\mathrm{Lan}^i_{ja}y^a=0$. Moreover, $\overset{\mathrm{nl}}{\imath_{\Can}}(\Gamma^{\mathrm{Ber}})=\mathring{N}$ and $\underset{\mathrm{nl}}{\imath_{\Can}}(\mathring{N})=\Gamma^{\mathrm{Ber}}$, where $\mathring{N}$ is the canonical nonlinear connection of $L$. With this information, the residue is computable trivially: 
	\[
	\begin{split}
		\Delta=\Gamma^{\mathrm{Ch}}-\underset{\mathrm{nl}}{\vd}(\overset{\mathrm{nl}}{\imath_{\Can}}(\Gamma^{\mathrm{Ch}}))&\equiv\left(\Gamma^{\mathrm{Ch}}\right)^i_{jk}-\left\{\left(\Gamma^{\mathrm{Ch}}\right)^i_{ja}y^a\right\}_{\cdot k} \\
		&=\left(\Gamma^{\mathrm{Ber}}\right)^i_{jk}+\mathrm{Lan}^i_{jk}-\left\{\left(\Gamma^{\mathrm{Ber}}\right)^i_{ja}y^a\right\}_{\cdot k} \\
		&=\left(\Gamma^{\mathrm{Ber}}\right)^i_{jk}+\mathrm{Lan}^i_{jk}-\left(\Gamma^{\mathrm{Ber}}\right)^i_{jk} \\
		&= \mathrm{Lan}^i_{jk}.
	\end{split}
	\]
\end{example}
\section{Including linear connections} \label{section linear}

It remains to add to our treatment and to \eqref{ladder connections} a next level of connections above the anisotropic ones. It is one of the most classical classes in semi-Finsler geometry: that of the linear connections on the vertical bundle $\VV A\rightarrow A$ \cite[\S 1.2]{AP}. First, a couple terminological points: 
\begin{itemize}
	\item We will avoid the name \emph{Finslerian connections}. This is because for some authors \cite{Dahl}, this means the data of a connection on $\VV A$ together with a nonlinear one $N$. Here, we wish to emphasize the independence of the constructions with respect to particular choices of $N$.
	
	\item For consistency, we will keep in mind the isomorphism \eqref{vertical isomorphism} and, in practice, will denote everything in terms of $\pi^*_A(\TT M)$ rather than $\VV A$.
	
	\item For $\mathcal{X}\in\mathfrak{X}(A)$, we will take the definition of being \emph{homogeneous of degree $\alpha\in\mathbb{R}$} from\footnote{Notice the notational differences: the $\alpha$ of \cite[(22)]{HPV2} would be our $\lambda$. Also, notice that in \cite[\S 5.1]{JSV1}, we chose a different convention for naming the homogeneity of elements of $\mathfrak{X}(A)$.
		 
		 About our terminology for homogeneous linear connections, recall that in \cite{JSV1} we called them just \emph{invariant by homotheties}. We prefer not to choose between the terms \emph{$0$-homogeneous} and \emph{$\left(-1\right)$-hom.}, as these connections have components of both kinds and there may be further arguments to highlight either of them.} \cite[Def. 6]{HPV2}. By $\mathrm{h}_{\alpha}\mathfrak{X}(A)\subset\mathfrak{X}(A)$, we will indicate the set of all the $\mathcal{X}$'s that are $\alpha$-homogeneous.
\end{itemize}
A \emph{linear connection} on any vector bundle $E$ over $A$ is an operator $\hat{\nabla}\colon\mathfrak{X}(A)\times\varGamma(E)\rightarrow\varGamma(E)$ satisfying the standard properties of $\Fun(A)$-linearity in the first variable and $\mathbb{R}$-linearity plus Leibniz rule in the second one. In what follows, linear connections will be defined on $E=\pi^*_A(\TT M)$, whose set of sections is $\varGamma(E)=\Ten_0^1(M_A)$, and will be assumed to be \emph{homogeneous}, meaning that they have well-defined restrictions $\hat{\nabla}\colon\mathrm{h}_{\alpha}\mathfrak{X}(A)\times\mathrm{h}_{\beta}\Ten_0^1(M_A)\rightarrow\mathrm{h}_{\alpha+\beta}\Ten_0^1(M_A)$ for each $\alpha,\beta\in\mathbb{R}$. 

We will denote the set of all these linear connections by $\mathscr{C}_{\mathrm{lin}}(A)$. Their definition as Koszul operators $\hat{\nabla}$ was taken for simplicity, but an affine bundle $\mathrm{C}(\VV A)\rightarrow A$ of which they are sections $\hat{\Gamma}$ could be constructed, by following a procedure analogous to that of \eqref{connection bundle}.
%\begin{definition}
%	By a \emph{linear connection (on $\pi^*_A(\TT M)$)}, we will mean a Koszul operator $\hat{\nabla}\colon\mathfrak{X}(A)\times\Ten_0^1(M_A)\rightarrow\Ten_0^1(M_A)$. These connections will be assumed to be \emph{homogeneous}, in that they have well-defined restrictions $\hat{\nabla}\colon\mathrm{h}_{\alpha}\mathfrak{X}(A)\times\mathrm{h}_{\beta}\Ten_0^1(M_A)\rightarrow\mathrm{h}_{\alpha+\beta}\Ten_0^1(M_A)$ for each $\alpha,\beta\in\mathbb{R}$, and we will denote the set of all of them by $\mathscr{C}_{\mathrm{lin}}(A)$.
%\end{definition}

\subsection{Using an auxiliary nonlinear connection}

What follows is really a summary the results of \cite[\S 5 and 6]{JSV1} for linear connections. These can be localized, so determining $\hat{\nabla}$ is equivalent to determining its coefficients on each natural chart $(A\cap\TT U,(x,y))$, namely
\begin{equation}
	\left(\hat{\Gamma}^1\right)^i_{jk}\partial_{x^i}:=\hat{\nabla}_{\partial_{x^j}}\partial_{x^k},\qquad\left(\hat{\Gamma}^2\right)^i_{jk}\partial_{x^i}:=\hat{\nabla}_{\partial_{y^j}}\partial_{x^k}.
	\label{linear connection coords}
\end{equation}
In this subsection, we assume that a nonlinear connection $N\in\mathscr{C}_{\mathrm{nl}}(A)$ is given. In this case, one can change the local basis $\left\{\partial_{x^i},\partial_{y^i}\right\}$  of $\mathfrak{X}(A)$ to $\left\{\delta_{x^i},\partial_{y^i}\right\}$, where
\[
\delta_{x^i}=\frac{\delta}{\delta x^i}:=\frac{\partial}{\partial x^i}-N^a_i\frac{\partial}{\partial y^a}.
\]
Here, $\left\{\partial_{y^i}\right\}$ is a basis for sections of $\VV A$ and $\left\{\delta_{x^i}\right\}$ is one for the sections of $\HH A$, this being the \emph{horizontal bundle} \cite[(13)]{JSV1}. Thus, $\hat{\nabla}$ is equivalent to the pair $(\Gamma,\Delta)$, where 
\begin{equation}
	\Gamma^i_{jk}\,\partial_{x^i}:=\hat{\nabla}_{\delta_{x^j}}\partial_{x^k},\qquad \Delta^i_{ jk}\,\partial_{x^i}:=\hat{\nabla}_{\partial_{y^j}}\partial_{x^k}=\left(\hat{\Gamma}^2\right)^i_{jk}\partial_{x^i}.
	\label{indentification linear 0}
\end{equation}
It is known that $\Gamma$ (the \emph{horizontal part of $\hat{\nabla}$ according to $N$}) is an anisotropic connection and $\Delta$ (its \emph{vertical part}, called \emph{vertical deviation} in \cite[Lem. and Def. 6.2.24]{KLS}) is a $\left(-1\right)$-homogeneous anisotropic tensor. Indeed, one can see \cite[Prop. 3]{JSV1}, where $\Gamma$ is denoted $\Gamma^\HH$ and $\Delta$ is denoted $\Gamma^\VV$; the linear connections with $\Delta=0$ are called \emph{vertically trivial} there (and \emph{vertically natural} in \cite{KLS}). The later role of $\Delta$, analogous to the objects of Def. \ref{residues}, will justify the following.
\begin{definition}
	Given $\hat{\nabla}\in\mathscr{C}_{\mathrm{lin}}(A)$, we call the corresponding $\Delta$ in \eqref{indentification linear 0} the \emph{residue} of $\hat{\nabla}$ with respect to $\mathscr{C}_{\mathrm{anis}}(A)$.
\end{definition}

Let us collect a couple computations for future referencing:
\begin{enumerate}
	\item The relation between the two ways of  expressing $\hat{\nabla}$ is as follows:
	\begin{equation}
		\Gamma^i_{jk}=\hat{\nabla}_{\partial_{x^j}}\partial_{x^k}-N_j^b\,\hat{\nabla}_{y^b}\partial_{x^k}=\left(\hat{\Gamma}^1\right)^i_{jk}-N^b_j\Delta_{ bk}^i.
		\label{relation gamma and hat}
	\end{equation}
%	As a consequence, the transformation law of the $\Gamma^i_{jk}$'s is precisely \eqref{transformation anisotropic}, whereas that for the $\Delta^i_{ jk}$'s is just 
%	\[
%	\wt{\Delta^i_{ jk}}=\frac{\partial\wt{x}^i}{\partial x^a}\frac{\partial x^b}{\partial\wt{x}^j}\frac{\partial x^c}{\partial\wt{x}^k}\Delta^a_{\,bc}.
%	\]
	
	\item For $\mathcal{X}=\overset{1}{\mathcal{X}^i}\,\partial_{x^i}+\overset{2}{\mathcal{X}^i}\,\partial_{y^i}=X^i\,\delta_{x^i}+Y^i\,\partial_{y^i}\in \mathfrak{X}(A)$ and $Z=Z^i\,\partial_{x^i}\in\Ten_0^1(M_A)$,
	\begin{equation}
		\begin{split}
			\hat{\nabla}_{\mathcal{X}}Z&=\left\{\overset{1}{\mathcal{X}^j}\frac{\partial Z^i}{\partial x^j}+\left(\hat{\Gamma}^1\right)^i_{jc}\overset{1}{\mathcal{X}^j}Z^c+\overset{2}{\mathcal{X}^j}\frac{\partial Z^i}{\partial y^j}+\left(\hat{\Gamma}^2\right)^i_{jc}\overset{2}{\mathcal{X}^j}Z^c\right\}\partial_{x^i}\\
			&=\left\{X^j\left(\frac{\delta Z^i}{\delta x^j}+\Gamma^i_{jc}Z^c\right)+Y^j\left(\frac{\partial Z^i}{\partial y^j}+\Delta^i_{ jc}Z^c\right)\right\}\partial_{x^i} \\
			&\in\Ten_0^1(M_A).
		\end{split}
	\label{hat nabla}
	\end{equation}

\end{enumerate}
\begin{proposition} \label{prop linear}
	There is a natural map
	\begin{equation}
		\hat{\nabla}\in\mathscr{C}_{\mathrm{lin}}(A)\longmapsto \Delta\in\mathrm{h}_{-1}\Ten^1_{2}(M_A)
		\label{second projection linear}
	\end{equation}
	given by the second part of \eqref{indentification linear 0}. Furthermore:
	\begin{enumerate} [label=\it{\roman*)}]
		\item \label{prop linear 1} In the presence of the fixed nonlinear connection $N$, there appears an identification
		\begin{equation}
			\begin{matrix}
				\mathscr{C}_{\mathrm{lin}}(A)&\equiv&\mathscr{C}_{\mathrm{anis}}(A)\times\mathrm{h}_{-1}\Ten^1_{2}(M_A), \\
				\hat{\nabla}&\equiv&(\Gamma,\Delta),
				
			\end{matrix}
		\label{identification linear}
		\end{equation}
		given by the entirety of \eqref{indentification linear 0}. Under it, \eqref{second projection linear} becomes the second projection. 
		
		\item \label{prop linear 2} The restriction of \eqref{identification linear}
		\begin{equation}
				\begin{matrix}
				\left\{\hat{\nabla}\in\mathscr{C}_{\mathrm{lin}}(A)\colon \Delta=0\right\}&\equiv&\mathscr{C}_{\mathrm{anis}}(A)\times\left\{0\right\} &\equiv&\mathscr{C}_{\mathrm{anis}}(A), \\
				\hat{\nabla}&\equiv&(\Gamma,0)&\equiv&\Gamma,
			\end{matrix}
		\label{vertically trivial}
		\end{equation}
		turns out to be independent of the chosen $N$.
	\end{enumerate} 
\end{proposition}
\begin{proof}
	Everything follows from the above remarks, but one can also see the proofs of \cite[Prop. 3 and Th. 3]{JSV1}.
\end{proof}
Now, we introduce operators that will play a role analogous to those of $\imath_{\Can}$ and $\vd$ in \eqref{ladder connections}, but to transition between $\mathscr{C}_{\mathrm{lin}}(A)$ and $\mathscr{C}_{\mathrm{anis}}(A)$.
\begin{corollary} \label{cor linear}
	Let us denote by $\overset{\mathrm{anis}}{\jmath_N}\colon\mathscr{C}_{\mathrm{lin}}(A)\rightarrow\mathscr{C}_{\mathrm{anis}}(A)$ the map $\hat{\nabla}\mapsto\Gamma$ in \eqref{identification linear}, and by $\underset{\mathrm{anis}}{\varrho}\colon\mathscr{C}_{\mathrm{anis}}(A)\rightarrow \mathscr{C}_{\mathrm{lin}}(A)$ the one that is well-defined by Prop. \ref{prop linear} \ref{prop linear 2}. Then, 
	\begin{equation}
		\overset{\mathrm{anis}}{\jmath_N}\circ\underset{\mathrm{anis}}{\varrho}=\mathrm{Id}_{\mathscr{C}_{\mathrm{anis}}(A)},
		\label{fake euler}
	\end{equation}
	the operator $\underset{\mathrm{anis}}{\varrho}$ is injective, $	\overset{\mathrm{anis}}{\jmath_N}$ is surjective and the 
	\begin{equation}
		\begin{matrix}
			\mathscr{C}_{\mathrm{lin}}(A)&\equiv&\Img(\underset{\mathrm{anis}}{\varrho})\times\mathrm{h}_{-1}\Ten^1_{2}(M_A), \\
			\hat{\nabla}&\equiv&(\underset{\mathrm{anis}}{\varrho}(\overset{\mathrm{anis}}{\jmath_N}(\hat{\nabla})),\hat{\nabla}-\underset{\mathrm{anis}}{\varrho}(\overset{\mathrm{anis}}{\jmath_N}(\hat{\nabla}))).
		\end{matrix}
	\label{decomposition linear}
	\end{equation}
	\begin{proof}
		The identity \eqref{fake euler} and the injectivity and surjectivity are obvious from Prop. \ref{prop linear} and how $\overset{\mathrm{anis}}{\jmath_N}$ and $\underset{\mathrm{anis}}{\varrho}$ are defined. For \eqref{decomposition linear}, we are just taking \eqref{identification linear}, identifying $\mathscr{C}_{\mathrm{anis}}(A)$ with its image and realizing that one obtains $\Delta$ as follows:
		\[
		\hat{\nabla}\equiv(\Gamma,\Delta)\longmapsto\overset{\mathrm{anis}}{\jmath_N}(\hat{\nabla})=\Gamma\longmapsto\underset{\mathrm{anis}}{\varrho}(\overset{\mathrm{anis}}{\jmath_N}(\hat{\nabla}))\equiv(\Gamma,0),
		\]
		\[
		\hat{\nabla}-\underset{\mathrm{anis}}{\varrho}(\overset{\mathrm{anis}}{\jmath_N}(\hat{\nabla}))\equiv(\Gamma,\Delta)-(\Gamma,0)=\Delta.
		\]
	\end{proof}
\end{corollary}
Due to this result, one can prolong \eqref{ladder connections} to the left:
\begin{equation}
	\xymatrix{ 
		\mathscr{C}_{\mathrm{lin}}(A) \ar@/^/[r]^{\overset{\mathrm{anis}}{\jmath_N}} &\mathscr{C}_{\mathrm{anis}}(A)\ar@/^/[l]^{\underset{\mathrm{anis}}{\varrho}}\ar@/^/[r]^{\overset{\mathrm{nl}}{\imath_{\Can}}} &  \mathscr{C}_{\mathrm{nl}}(A)\ar@/^/[l]^{\underset{\mathrm{nl}}{\vd}}\ar@/^/[r]^{\overset{\mathrm{spr}}{\imath_{\Can}}} & \mathscr{C}_{\mathrm{spr}}(A).\ar@/^/[l]^{\underset{\mathrm{spr}}{\vd}}}
	\label{ladder connections linear}
\end{equation}
As announced, this is formally consistent with Cor. \ref{corollary connections} despite the very different natures of $\overset{\mathrm{anis}}{\jmath_N}$ and $\underset{\mathrm{anis}}{\varrho}$ from those of the $\imath_{\Can}$'s and $\vd$'s, resp.

\subsection{Intrinsically}

We now aim to complete the landscape of correspondences among the connection-type objects. In order to do so, let us see that, to some extent, one can replace the map $\overset{\mathrm{anis}}{\jmath_N}$ by another one that does not depend on any auxiliary nonlinear connection. 

The strategy will be to make $\hat{\nabla}\in\mathscr{C}_{\mathrm{lin}}(A)$ produce itself a natural nonlinear connection, and then evaluate $\overset{\mathrm{anis}}{\jmath_N}$ with respect to it. For this, we turn to the regularity conditions of \cite{Min23}. One says that $\hat{\nabla}$ is \emph{regular} if the restriction $\hat{\nabla}\Can\colon\VV A\rightarrow\pi^*_A(\TT M)$ is an isomorphism of bundles over $A$, and that it is \emph{strongly regular} if $\left.\hat{\nabla}\Can\right|_{\VV A}$ is the identity when considered with codomain $\VV A$ through \eqref{vertical isomorphism}. In either case, putting $\HH_v A:=\Ker(\left.\hat{\nabla}\Can\right|_{\TT_v A})$ for $v\in A$ defines a horizontal bundle, and therefore must correspond to a unique nonlinear connection $N\in\mathscr{C}_{\mathrm{nl}}(A)$.

We shall write $\mathscr{C}_{\mathrm{lin}}^{\mathrm{reg}}(A)$ for the set of regular linear connections. They had been called \emph{good connections} in \cite[Def. 1.2.2]{AP}, below which was the explicit computation of the nonlinear connection induced by $\hat{\nabla}$. We reproduce it for completeness, starting with the following expression from \eqref{hat nabla}:
\[
\hat{\nabla}_{\mathcal{X}}\Can=\left\{\delta^i_j\overset{2}{\mathcal{X}^j}+\left(\hat{\Gamma}^2\right)^i_{jc}y^c\overset{2}{\mathcal{X}^j}+\left(\hat{\Gamma}^1\right)^i_{jc}y^c\overset{1}{\mathcal{X}^j}\right\}\partial_{x^i}.
\]
 With this, one easily sees that
\begin{equation}
	\mathscr{C}_{\mathrm{lin}}^{\mathrm{reg}}(A)=\left\{\hat{\nabla}\in\mathscr{C}_{\mathrm{lin}}(A)\colon\left(\delta^i_j+\left(\hat{\Gamma}^2\right)^i_{jc}y^c\right)_{n\times n}\text{ invertible everywhere}\right\}.
	\label{regular connections}
\end{equation}
For $\hat{\nabla}\in\mathscr{C}_{\mathrm{lin}}^{\mathrm{reg}}(A)$, denote 
\begin{equation}
	\left(B^i_j\right):=\left(\delta^i_j+\left(\hat{\Gamma}^2\right)^i_{jc}y^c\right)^{-1},
	\label{definition B}
\end{equation}
so that
\begin{equation}
	\Ker(\hat{\nabla}\Can)\equiv\left\{\overset{2}{\mathcal{X}^i}+B^i_a\left(\hat{\Gamma}^1\right)^a_{jc}y^c\overset{1}{\mathcal{X}^j}=0,\quad i\in\left\{1,\ldots,n\right\}\right\}.
	\label{description horizontal}
\end{equation}
It is well-known \cite[(13)]{JSV1} that for $v\in A$,
\[
\HH_vA=\mathrm{Span}\left\{\left.\partial_{x^i}\right|_v-N_i^a(v)\left.\partial_{y^a}\right|_v\right\},
\]
and by choosing $\overset{1}{\mathcal{X}^j}=\delta^j_i$ for each $i$ in \eqref{description horizontal}, it follows that 
\begin{equation}
	N_i^a(v)=B^a_b(v)\left(\hat{\Gamma}^1\right)^b_{ic}(v)\,y^c(v).
	\label{nonlinear from linear}
\end{equation}
(Strongly regular case: $B^i_j=\delta^i_j$ and $N^i_j=\left(\hat{\Gamma}^1\right)^i_{jc}y^c$.)

The next theorem rounds up the correspondences between connections by using \eqref{nonlinear from linear} to produce an anisotropic one.
\begin{theorem} \label{cor linear intrinsic}
 There is a well-defined map
	\[
	\overset{\mathrm{anis}}{\jmath}\colon\mathscr{C}_{\mathrm{lin}}^{\mathrm{reg}}(A)\longrightarrow\mathscr{C}_{\mathrm{anis}}(A)
	\]
	by taking $\hat{\nabla}$ to (in the notation of Cor. \ref{cor linear}) $\overset{\mathrm{anis}}{\jmath_N}(\hat{\nabla})$ for the $N$ of \eqref{nonlinear from linear}. (I.e., if $\hat{\nabla}$ is given by \eqref{linear connection coords}, then $\overset{\mathrm{anis}}{\jmath}(\hat{\nabla})=\Gamma$ with 
	\[
	\Gamma^i_{jk}=\left(\hat{\Gamma}^1\right)^i_{jk}-B^b_a\left(\hat{\Gamma}^1\right)^a_{jc}y^c\left(\hat{\Gamma}^2\right)^i_{bk},
	\]
	where $\left(B^i_j\right)$ is the inverse matrix of $\left(\delta^i_j+\left(\hat{\Gamma}^2\right)^i_{jc}y^c\right)$.) Moreover:
	\begin{enumerate} [label=\it{\roman*)}]
		\item It holds that 
		\[
		\Img(\underset{\mathrm{anis}}{\varrho})\subset\mathscr{C}_{\mathrm{lin}}^{\mathrm{reg}}(A),\qquad\overset{\mathrm{anis}}{\jmath}\circ\underset{\mathrm{anis}}{\varrho}=\mathrm{Id}_{\mathscr{C}_{\mathrm{anis}}(A)},
		\]
		so $\overset{\mathrm{anis}}{\jmath}$ is surjective.
		
		\item \label{cor linear intrinsic 2} There is a bijective correspondence
		\[
		\begin{matrix}
			\mathscr{C}_{\mathrm{lin}}^{\mathrm{reg}}(A)&\equiv&\mathscr{C}_{\mathrm{anis}}(A)\times\left\{\Delta\in\mathrm{h}_{-1}\Ten^1_2(M_A)\colon\left(\delta^i_j+\Delta^i_{jc}y^c\right)\mathrm{\; invertible \; everywhere}\right\},\\
			\hat{\nabla}&\equiv&(\Gamma,\Delta).
		\end{matrix}
		\]
		To be precise, from left to right, $\Gamma=\overset{\mathrm{anis}}{\jmath}(\hat{\nabla})$ and $\Delta$ is given by \eqref{second projection linear}, whereas from right to left, $\hat{\nabla}$ is obtained from $(\Gamma,\Delta)$ by applying Prop. \ref{prop linear} \ref{prop linear 1} for $N=\overset{\mathrm{nl}}{\imath_{\Can}}(\Gamma)$.
	\end{enumerate}
\end{theorem}
\begin{proof}
	The first assertions follow from the above remarks, including \eqref{relation gamma and hat} and \eqref{nonlinear from linear} for the relation between the components of $\hat{\nabla}$ and $\overset{\mathrm{anis}}{\jmath}(\hat{\nabla})$.
	\begin{enumerate} [label=\it{\roman*)}]
		\item From the definition of $\underset{\mathrm{anis}}{\varrho}$ (see Cor. \ref{cor linear}), we know that its image is the set of vertically trivial linear connections, i.e., the left hand side of \eqref{vertically trivial}, clearly contained in \eqref{regular connections}. What is more, consider $\Gamma\in	\mathscr{C}_{\mathrm{anis}}(A)$,  $	\underset{\mathrm{anis}}{\varrho}\Gamma\in\mathscr{C}_{\mathrm{lin}}^{\mathrm{reg}}(A)$ and the $N\in\mathscr{C}_{\mathrm{nl}}(A)$ associated with $\underset{\mathrm{anis}}{\varrho}\Gamma$ via \eqref{nonlinear from linear}. In the presence of $N$ (or any other nonlinear connection, for that matter), Prop. \ref{prop linear} \ref{prop linear 1} gives us $\underset{\mathrm{anis}}{\varrho}\Gamma\equiv(\Gamma,0)$, and so
		\[
		\overset{\mathrm{anis}}{\jmath}(\underset{\mathrm{anis}}{\varrho}\Gamma)=\Gamma.
		\]
		
		\item Let us see that the two described maps compose to the identity.
		
		First, we obtain $\Gamma$ and $\Delta$ from $\hat{\nabla}$. This requires of the nonlinear connection of components $B^i_a\left(\hat{\Gamma}^1\right)^a_{jc}y^c$: the relation \eqref{relation gamma and hat} in this case tells us that 
		\[
		\Gamma^i_{jk}=\left(\hat{\Gamma}^1\right)^i_{jk}-B^b_a\left(\hat{\Gamma}^1\right)^a_{jc}y^c\Delta_{ bk}^i.
		\]
		Now we compute $N=\overset{\mathrm{nl}}{\imath_{\Can}}(\Gamma)$:
		\begin{equation}
			\begin{split}
				N^i_j=\Gamma^i_{jc}y^c=\left(\hat{\Gamma}^1\right)^i_{jc}y^c-B^b_a\left(\hat{\Gamma}^1\right)^a_{jc}y^c\Delta_{ bd}^i y^d&=\left(\hat{\Gamma}^1\right)^a_{jc}y^c\left(\delta^i_a-B^b_a\left(\hat{\Gamma}^2\right)_{bd}^i y^d\right) \\
				&=B^i_a\left(\hat{\Gamma}^1\right)^a_{jc}y^c.
			\end{split}
		\label{two nonlinear connections coincide}
		\end{equation}
		by using  $\Delta^i_{ jk}=\left(\hat{\Gamma}^2\right)_{jk}^i$ and \eqref{definition B}. So, $N$ is the nonlinear connection with which we started. But it is according to this one that $\hat{\nabla}\equiv(\Gamma,\Delta)$, so when we apply the right-to-left map of Th. \ref{cor linear intrinsic} \ref{cor linear intrinsic 2}, we recover $\hat{\nabla}$.
		
		Second, we obtain $\hat{\nabla}$ from $(\Gamma,\Delta)$, by means of $N=\overset{\mathrm{nl}}{\imath_{\Can}}(\Gamma)$. The components of $\hat{\nabla}$ are determined by \eqref{indentification linear 0} and \eqref{relation gamma and hat}, and the fact that $\left(\delta^i_j+\left(\hat{\Gamma}^2\right)^i_{jc}y^c\right)=\left(\delta^i_j+\Delta^i_{ jc}y^c\right)$ guarantees that we can compute $B^i_a\left(\hat{\Gamma}^1\right)^a_{jc}y^c$. When doing so, one recovers $N^i_j$, as happened in \eqref{two nonlinear connections coincide}. From here, the same argument as above shows that when applying the left-to-right map of Th. \ref{cor linear intrinsic} \ref{cor linear intrinsic 2} to $\hat{\nabla}$, we recover $\Gamma=\overset{\mathrm{anis}}{\jmath}(\hat{\nabla})$, while $\Delta$ is by construction the vertical part \eqref{second projection linear} of $\hat{\nabla}$.
	\end{enumerate}
\end{proof}
This result provides another consistent prolongation of \eqref{ladder connections}:
\begin{equation}
	\xymatrix{ 
		\mathscr{C}_{\mathrm{lin}}^{\mathrm{reg}}(A) \ar@/^/[r]^{\overset{\mathrm{anis}}{\jmath}} &\mathscr{C}_{\mathrm{anis}}(A)\ar@/^/[l]^{\underset{\mathrm{anis}}{\varrho}}\ar@/^/[r]^{\overset{\mathrm{nl}}{\imath_{\Can}}} &  \mathscr{C}_{\mathrm{nl}}(A)\ar@/^/[l]^{\underset{\mathrm{nl}}{\vd}}\ar@/^/[r]^{\overset{\mathrm{spr}}{\imath_{\Can}}} & \mathscr{C}_{\mathrm{spr}}(A).\ar@/^/[l]^{\underset{\mathrm{spr}}{\vd}}}
	\label{ladder connections 2}
\end{equation}

Our previous work \cite[\S 6.2]{JSV1} contains an account of the most classical linear connections attached to a semi-Finsler Lagrangian $L$: the Berwald, Hashiguchi, Chern-Rund and Cartan connections. It is important to keep in mind that all of them are strongly regular. (One can see this by recalling that for the Berwald and Chern ones, $\Delta=0$, while for the Hashiguchi  and Cartan ones, $\Delta$ is the Cartan tensor; in all cases, $\Delta^i_{ ja}y^a=0$, which is the strong regularity condition.) One may check that the nonlinear connection \eqref{nonlinear from linear} produced in all four cases is the \emph{canonical nonlinear connection}
\[
\mathring{N}:=\underset{\mathrm{spr}}{\vd}\mathring{G},\qquad \mathring{G}^i:=\frac{1}{4}g^{ic}\left(\frac{\partial g_{cb}}{\partial x^a}+\frac{\partial g_{ac}}{\partial x^b}-\frac{\partial g_{ab}}{\partial x^c}\right)y^ay^b.
\]
So, this is the connection that will be used to decompose the four $\hat{\nabla}$'s as $(\Gamma,\Delta)$ when computing $\overset{\mathrm{anis}}{\jmath}$. Going back to \cite[\S 6.2]{JSV1}, we see that $\overset{\mathrm{anis}}{\jmath}(\hat{\nabla})$ is the Berwald anisotropic connection when $\hat{\nabla}$ is Berwald's or Hashiguchi's, and it is the Chern anisotropic connction when $\hat{\nabla}$ is Chern's or Cartan's. These were known results, but here we have derived them in a language compatible with the rest of correspondences between connection-type objects in \S \ref{section connections}. We are also refining the viewpoint of \cite{JSV1}, since here the auxiliary nonlinear connection is not assumed from the beginning, but rather is derived from $\hat{\nabla}$.

\section{Consequences for variational problems} \label{variational problems}

To end this article, we turn our attention to anisotropic extensions of general relativity, including the Lorentz-Finsler ones, as mentioned in \S \ref{section metrics 1}. Concretely, we focus on the extensions that admit a variational formulation and their comparison. As discussed in \S \ref{introduction}, the examples PWHV \cite{PW,HPV}, JSV \cite{JSV2} and GM \cite{GPMin} are formulated for different metric-type and connection-type objects. The main point here is that if any kind of comparison is to be done between these theories, first one should attempt to put them on a common ground. We shall see that on the metric part there are some obstructions to this, but on the affine ladder \eqref{ladder connections 2} (or \eqref{ladder connections linear}) it is completely feasible.

Recall the sets $\mathscr{M}_{\mathrm{s-F}}$ (semi-Finsler Lagrangians), $\mathscr{M}_{\mathrm{Lt}}$ (Legendre transformations), $\mathscr{M}_{\mathrm{anis}}$ (anisotropic metrics), $\mathscr{C}_{\mathrm{spr}}$ (sprays), $\mathscr{C}_{\mathrm{nl}}$ (nonlinear connections), $\mathscr{C}_{\mathrm{anis}}$ (anisotropic connections) and $\mathscr{C}_{\mathrm{lin}}^{\mathrm{reg}}$ (regular linear connections). The theories of our interest are defined on either one of these or a product of two of them. (For instance, $\mathscr{M}_{\mathrm{s-F}}$ for \cite{PW,HPV} $\mathscr{M}_{\mathrm{s-F}}\times\mathscr{C}_{\mathrm{nl}}$ for \cite{JSV2} and $\mathscr{M}_{\mathrm{anis}}\times\mathscr{C}_{\mathrm{lin}}^{\mathrm{reg}}$, though via frames, for \cite{GPMin}.) However, we will only care about the dependence on each variable separately, thinking of the rest as fixed if needed. Thus:
\begin{itemize}
	\item We will refer as an \emph{(action) functional} to any map $\mathscr{S}\colon\mathscr{X}\rightarrow\mathbb{R}$ in which $\mathscr{X}$ is one of the above seven sets. 
	
	\item This $\mathscr{X}$ will also determine the class of \emph{variations} required to make sense of the critical point problem, when $\mathscr{S}$ is appropriately differentiable.
	
	\item We will obtain results in which new functionals are produced from $\mathscr{S}$ on different levels of \eqref{ladder connections 2}, or on $\mathscr{M}_{\mathrm{s-F}}$, $\mathscr{M}_{\mathrm{Lt}}$ or $\mathscr{M}_{\mathrm{anis}}$. In practise, $\mathscr{S}$ will be an integral of some Lagrangian density $\Lambda$. Then, it is important to keep in mind that new Lagrangian densities are being produced from $\Lambda$ on the aforementioned domains.
\end{itemize}

Let $\mathscr{S}$ be a functional defined on $\mathscr{M}_{\mathrm{anis}}$. There is a canonical way of making sense of its restriction to semi-Finsler Lagrangians:
\begin{equation}
	\left.\mathscr{S}\right|_{\mathscr{M}_{\mathrm{s-F}}}\colon\mathscr{M}_{\mathrm{s-F}}\longrightarrow\mathbb{R},\qquad\left.\mathscr{S}\right|_{\mathscr{M}_{\mathrm{s-F}}}[L]:=\mathscr{S}[\frac{1}{2}\,\underset{1}{\vd}\underset{2}{\vd}L].
	\label{restriction to pseudo-Finsler}
\end{equation}
On the other hand, when attempting to the define the restriction of $\mathscr{S}$ to Legendre transformations $\ell\in\mathscr{M}_{\mathrm{Lt}} $, one runs into the problem described in \S \ref{symmetry}. Namely, the symmetry of $\vd\ell$ has to be imposed, which is equivalent to $\ell=\vd L$ for a Lagrangian $L$, so one is back to \eqref{restriction to pseudo-Finsler}. If, instead of this, the functional was originally defined for Legendre transformations, one could still make sense of its restriction to semi-Finsler Lagrangians: $\left.\mathscr{S}\right|_{\mathscr{M}_{\mathrm{s-F}} }[L]:=\mathscr{S}[\underset{2}{\vd}L]$. 

The opposite process, extension of functionals for metric-type objects, fails in principle. Indeed, say that now we are given $\mathscr{S}_0\colon\mathscr{M}_{\mathrm{s-F}} \rightarrow\mathbb{R}$ and we want to evaluate it at an arbitrary anisotropic metric $g$. It does not make sense to write $\mathscr{S}_0[\overset{2}{\imath_{\Can}}\overset{1}{\imath_{\Can}}\,g]$, for $\overset{2}{\imath_{\Can}}\overset{1}{\imath_{\Can}}\,g\left(:=g(\Can,\Can)\right)$ may not be a semi-Finsler Lagrangian due to the regularity issues in \S \ref{regularity}. (Same if we want to extend $\mathscr{S}_0$ to $\mathscr{M}_{\mathrm{Lt}} $, or from $\mathscr{M}_{\mathrm{Lt}} $ to $\mathscr{M}_{\mathrm{anis}} $.)

The good news is that this symmetry and regularity problems when transitioning between metric-type objects are absent for connection-type ones. So, one will be able to restrict and extend functionals between any two levels of the ladders \eqref{ladder connections} and \eqref{ladder connections 2}.
\begin{theorem} \label{functionals}
	Let $\mathscr{S}\colon\mathscr{C}_{\mathrm{lin}}^{\mathrm{reg}} \rightarrow\mathbb{R}$ and $\mathscr{S}_0\colon\mathscr{C}_{\mathrm{anis}} \rightarrow\mathbb{R}$ be functionals. There is a natural way to define the restriction of $\mathscr{S}$ to $\mathscr{C}_{\mathrm{anis}} $, namely
	\[
	\left.\mathscr{S}\right|_{\mathscr{C}_{\mathrm{anis}} }\colon\mathscr{C}_{\mathrm{anis}} \longrightarrow\mathbb{R},\qquad\left.\mathscr{S}\right|_{\mathscr{C}_{\mathrm{anis}} }[\Gamma]:=\mathscr{S}[\underset{\mathrm{anis}}{\varrho}\Gamma],
	\]
	and one to extend $\mathscr{S}_0$ to $\mathscr{C}_{\mathrm{lin}}^{\mathrm{reg}} $, namely
	\[
	\mathscr{S}_1\colon\mathscr{C}_{\mathrm{lin}}^{\mathrm{reg}} \longrightarrow\mathbb{R},\qquad\mathscr{S}_1[\hat{\nabla}]:=\mathscr{S}_0[\overset{\mathrm{anis}}{\jmath}(\hat{\nabla})].
	\]
	Furthermore: 
	\begin{enumerate} [label=\it{\roman*)}]
		\item \label{functionals 1} There is a natural modification of $\mathscr{S}$ defined by 
		\[
		\wt{\mathscr{S}}\colon\mathscr{C}_{\mathrm{lin}}^{\mathrm{reg}} \longrightarrow\mathbb{R},\qquad\wt{\mathscr{S}}[\hat{\nabla}]:=\left.\mathscr{S}\right|_{\mathscr{C}_{\mathrm{anis}} }[\overset{\mathrm{anis}}{\jmath}(\hat{\nabla})]=\mathscr{S}[\underset{\mathrm{anis}}{\varrho}(\overset{\mathrm{anis}}{\jmath}(\hat{\nabla}))].
		\]
		Under the identification in Th. \ref{cor linear intrinsic} \ref{cor linear intrinsic 2}, this becomes a functional on $\mathscr{C}_{\mathrm{anis}} \times\left\{\Delta\in\mathrm{h}_{-1}\Ten^1_2 \colon\left(\delta^i_j+\Delta^i_{jc}y^c\right)\mathrm{\; invertible \; everywhere}\right\}$ invariant to the maps $(\Gamma,\Delta)\mapsto(\Gamma,\Delta+\Delta^\prime)$ for every admissible $\Delta^\prime$.
		
		\item \label{functionals 2} If an auxiliary nonlinear connection $N$ is given, one can extend $\mathscr{S}_0$ to $\mathscr{C}_{\mathrm{lin}} $ in a different way:
		\[
		\mathscr{S}_{2,N}\colon\mathscr{C}_{\mathrm{lin}} \longrightarrow\mathbb{R},\qquad\mathscr{S}_{2,N}[\hat{\nabla}]:=\mathscr{S}_0[\overset{\mathrm{anis}}{\jmath_N}(\hat{\nabla})].
		\]
		Accordingly, the alternative modification of $\mathscr{S}$
		\[
		\wt{\mathscr{S}}_{2,N}\colon\mathscr{C}_{\mathrm{lin}} \longrightarrow\mathbb{R},\qquad\wt{\mathscr{S}}_{2,N}[\hat{\nabla}]=\mathscr{S}[\underset{\mathrm{anis}}{\varrho}(\overset{\mathrm{anis}}{\jmath_N}(\hat{\nabla}))], 
		\]
		becomes under the identification in Prop. \ref{prop linear} \ref{prop linear 1} a functional on $\mathscr{C}_{\mathrm{anis}} \times\mathrm{h}_{-1}\Ten^1_{2} $ invariant to $(\Gamma,\Delta)\mapsto(\Gamma,\Delta+\Delta^\prime)$ for every $\Delta^\prime\in \mathrm{h}_{-1}\Ten^1_{2} $.
	\end{enumerate}
\end{theorem}
\begin{proof}
	Everything here is obvious, considering that the injective $\underset{\mathrm{anis}}{\varrho}$ allows one to regard $\mathscr{C}_{\mathrm{anis}} $ as a subset of $\mathscr{C}^{\mathrm{reg}}_{\mathrm{lin}} $ and that $\overset{\mathrm{anis}}{\jmath}$ projects the latter onto the former. For \ref{functionals 1}, recall that the map $\underset{\mathrm{anis}}{\varrho}\circ\overset{\mathrm{anis}}{\jmath}$ destroys the residue as in \eqref{destroying defects}, i.e., it is $(\Gamma,\Delta)\mapsto(\Gamma,0)$. The proof of \ref{functionals 2} is formally identical.
\end{proof}
\begin{remark} \label{functionals otro}
	By the same considerations but now based on Cor. \ref{corollary connections} instead of the theory of \S \ref{section linear}, we get two analogous results for the transitions $\xymatrix{ 
		\mathscr{C}_{\mathrm{anis}}\ar@/^/[r] &  \mathscr{C}_{\mathrm{nl}}\ar@/^/[l]}$
	and $\xymatrix{ 
		  \mathscr{C}_{\mathrm{nl}}\ar@/^/[r] & \mathscr{C}_{\mathrm{spr}}.\ar@/^/[l]}
	$ Their statements are obtained by taking Th. \ref{functionals}, except for its item \ref{functionals 2}, and doing the some replacements. For instance, for the first result, one replaces $(\mathscr{C}^{\mathrm{reg}}_{\mathrm{lin}} ,\mathscr{C}_{\mathrm{anis}} )$ by $(\mathscr{C}_{\mathrm{anis}} ,\mathscr{C}_{\mathrm{nl}} )$;  $(\Gamma,\underset{\mathrm{anis}}{\varrho})$ by $(N,\underset{\mathrm{nl}}{\vd})$; $(\hat{\nabla},\overset{\mathrm{anis}}{\jmath})$ by $(\Gamma,\overset{\mathrm{nl}}{\imath_{\Can}})$; and the last paragraph of Th. \ref{functionals} \ref{functionals 1} by the following: ``\emph{Under \eqref{referencia descomposicion}, this becomes a functional on $\mathscr{C}_{\mathrm{nl}} \times\Ker(\overset{1}{\imath_{\Can}})$ invariant to the maps $(N,\Delta)\mapsto(N,\Delta+\Delta^\prime)$ for every $\Delta^\prime\in \Ker(\overset{1}{\imath_{\Can}})$}".

\end{remark}

	As a conclusion, by combining Th. \ref{functionals} and Rem. \ref{functionals otro}, one can take any functional defined on any level of the ladder \eqref{ladder connections 2} or \eqref{ladder connections linear} and redefine it to live on any other level. Each time that one lowers the level, one obtains a modification of the original functional as in Th. \ref{functionals} \ref{functionals 1} (say, a sort of ``gauge symmetrization"). Consequently, the corresponding class of variations effectively gets reduced. On the opposite extreme, each time one raises the level, more degrees of freedom and variations are permitted. This may make that some critical points for variations on the lower level stop being critical for the new variations.
	
This viewpoint provides at least an heuristic explanation for results such as the final one of \cite{GPMin}. There, the only vacuum solutions of the theory are necessarily classical Lorentzian metrics; in particular, non-quadratic Lorentzian norms \cite[Def. 3.1]{JS19} would not be interpreted as vacuum states. (These norms would not have critical properties as strong as those of a scalar product.) Our approach suggests that these norms will naturally be such vacuum solutions if, instead, one stays at one of the lower levels $\mathscr{C}_{\mathrm{anis}} $, $\mathscr{C}_{\mathrm{nl}} $ or $\mathscr{C}_{\mathrm{spr}} $ of \eqref{ladder connections 2}.

\bibliography{sn-bibliography}% common bib file


\begin{thebibliography}{27}
\ifx \bisbn   \undefined \def \bisbn  #1{ISBN #1}\fi
\ifx \binits  \undefined \def \binits#1{#1}\fi
\ifx \bauthor  \undefined \def \bauthor#1{#1}\fi
\ifx \batitle  \undefined \def \batitle#1{#1}\fi
\ifx \bjtitle  \undefined \def \bjtitle#1{#1}\fi
\ifx \bvolume  \undefined \def \bvolume#1{\textbf{#1}}\fi
\ifx \byear  \undefined \def \byear#1{#1}\fi
\ifx \bissue  \undefined \def \bissue#1{#1}\fi
\ifx \bfpage  \undefined \def \bfpage#1{#1}\fi
\ifx \blpage  \undefined \def \blpage #1{#1}\fi
\ifx \burl  \undefined \def \burl#1{\textsf{#1}}\fi
\ifx \doiurl  \undefined \def \doiurl#1{\url{https://doi.org/#1}}\fi
\ifx \betal  \undefined \def \betal{\textit{et al.}}\fi
\ifx \binstitute  \undefined \def \binstitute#1{#1}\fi
\ifx \binstitutionaled  \undefined \def \binstitutionaled#1{#1}\fi
\ifx \bctitle  \undefined \def \bctitle#1{#1}\fi
\ifx \beditor  \undefined \def \beditor#1{#1}\fi
\ifx \bpublisher  \undefined \def \bpublisher#1{#1}\fi
\ifx \bbtitle  \undefined \def \bbtitle#1{#1}\fi
\ifx \bedition  \undefined \def \bedition#1{#1}\fi
\ifx \bseriesno  \undefined \def \bseriesno#1{#1}\fi
\ifx \blocation  \undefined \def \blocation#1{#1}\fi
\ifx \bsertitle  \undefined \def \bsertitle#1{#1}\fi
\ifx \bsnm \undefined \def \bsnm#1{#1}\fi
\ifx \bsuffix \undefined \def \bsuffix#1{#1}\fi
\ifx \bparticle \undefined \def \bparticle#1{#1}\fi
\ifx \barticle \undefined \def \barticle#1{#1}\fi
\bibcommenthead
\ifx \bconfdate \undefined \def \bconfdate #1{#1}\fi
\ifx \botherref \undefined \def \botherref #1{#1}\fi
\ifx \url \undefined \def \url#1{\textsf{#1}}\fi
\ifx \bchapter \undefined \def \bchapter#1{#1}\fi
\ifx \bbook \undefined \def \bbook#1{#1}\fi
\ifx \bcomment \undefined \def \bcomment#1{#1}\fi
\ifx \oauthor \undefined \def \oauthor#1{#1}\fi
\ifx \citeauthoryear \undefined \def \citeauthoryear#1{#1}\fi
\ifx \endbibitem  \undefined \def \endbibitem {}\fi
\ifx \bconflocation  \undefined \def \bconflocation#1{#1}\fi
\ifx \arxivurl  \undefined \def \arxivurl#1{\textsf{#1}}\fi
\csname PreBibitemsHook\endcsname

%%% 1
\bibitem[\protect\citeauthoryear{O'Neill}{1983}]{ON}
\begin{bbook}
\bauthor{\bsnm{O'Neill}, \binits{B.}}:
\bbtitle{Semi-{R}iemannian Geometry with Applications to Relativity}.
\bsertitle{Pure and Applied Mathematics series},
vol. \bseriesno{103}.
\bpublisher{Academic Press},
\blocation{London}
(\byear{1983})
\end{bbook}
\endbibitem

%%% 2
\bibitem[\protect\citeauthoryear{Javaloyes et~al.}{2022a}]{JSV1}
\begin{bchapter}
\bauthor{\bsnm{Javaloyes}, \binits{M.A.}},
\bauthor{\bsnm{S\'anchez}, \binits{M.}},
\bauthor{\bsnm{Villaseñor}, \binits{F.F.}}:
\bctitle{Anisotropic connections and parallel transport in {F}insler
  spacetimes}.
In: \bbtitle{Developments in {L}orentzian Geometry}.
\bsertitle{Springer Proceedings in Mathematics \& Statistics},
vol. \bseriesno{389},
pp. \bfpage{175}--\blpage{206}
(\byear{2022}).
\bcomment{(Proceedings of GeLoCor 2021)}
\end{bchapter}
\endbibitem

%%% 3
\bibitem[\protect\citeauthoryear{Javaloyes et~al.}{2022b}]{JSV2}
\begin{barticle}
\bauthor{\bsnm{Javaloyes}, \binits{M.A.}},
\bauthor{\bsnm{S\'anchez}, \binits{M.}},
\bauthor{\bsnm{Villaseñor}, \binits{F.F.}}:
\batitle{The {E}instein-{H}ilbert-{P}alatini formalism in pseudo-{F}insler
  geometry}.
\bjtitle{Adv. Theor. Math. Phys.}
\bvolume{26}(\bissue{10}),
\bfpage{3563}--\blpage{3631}
(\byear{2022})
\end{barticle}
\endbibitem

%%% 4
\bibitem[\protect\citeauthoryear{Matthias}{1980}]{Mts}
\begin{botherref}
\oauthor{\bsnm{Matthias}, \binits{H.-H.}}:
Zwei verallgemeinerungen eines satzes von gromoll und meyer.
Phd thesis,
Universitat Bonn Mathematisches Institut,
Bonn
(1980)
\end{botherref}
\endbibitem

%%% 5
\bibitem[\protect\citeauthoryear{Javaloyes}{2019}]{Jav19}
\begin{barticle}
\bauthor{\bsnm{Javaloyes}, \binits{M.A.}}:
\batitle{Anisotropic tensor calculus}.
\bjtitle{Int. J. Geom. Methods Mod. Phys.}
\bvolume{16}(\bissue{{No. supp.02}}),
\bfpage{1941001}
(\byear{2019})
\end{barticle}
\endbibitem

%%% 6
\bibitem[\protect\citeauthoryear{Pfeifer and Wohlfart}{2012}]{PW}
\begin{barticle}
\bauthor{\bsnm{Pfeifer}, \binits{C.}},
\bauthor{\bsnm{Wohlfart}, \binits{M.N.R.}}:
\batitle{{Finsler geometric extension of Einstein gravity}}.
\bjtitle{Phys. Rev. D}
\bvolume{85}(\bissue{6}),
\bfpage{064009}
(\byear{2012})
\end{barticle}
\endbibitem

%%% 7
\bibitem[\protect\citeauthoryear{Hohmann et~al.}{2019}]{HPV}
\begin{barticle}
\bauthor{\bsnm{Hohmann}, \binits{M.}},
\bauthor{\bsnm{Pfeifer}, \binits{C.}},
\bauthor{\bsnm{Voicu}, \binits{N.}}:
\batitle{{Finsler gravity action from variational completion}}.
\bjtitle{Phys. Rev. D}
\bvolume{100}(\bissue{6}),
\bfpage{064035}
(\byear{2019})
\end{barticle}
\endbibitem

%%% 8
\bibitem[\protect\citeauthoryear{García-Parrado and Minguzzi}{2022}]{GPMin}
\begin{barticle}
\bauthor{\bsnm{García-Parrado}, \binits{A.}},
\bauthor{\bsnm{Minguzzi}, \binits{E.}}:
\batitle{An anisotropic gravity theory}.
\bjtitle{Gen. Relativity Gravitation}
\bvolume{54},
\bfpage{150}
(\byear{2022})
\end{barticle}
\endbibitem

%%% 9
\bibitem[\protect\citeauthoryear{Szilasi et~al.}{2013}]{KLS}
\begin{bbook}
\bauthor{\bsnm{Szilasi}, \binits{J.}},
\bauthor{\bsnm{Lovas}, \binits{R.L.}},
\bauthor{\bsnm{Kert\'esz}, \binits{D.C.}}:
\bbtitle{Connections, Sprays and {F}insler Structures}.
\bpublisher{World Scientific},
\blocation{Singapore}
(\byear{2013})
\end{bbook}
\endbibitem

%%% 10
\bibitem[\protect\citeauthoryear{Bucataru and Miron}{2007}]{BM}
\begin{bbook}
\bauthor{\bsnm{Bucataru}, \binits{I.}},
\bauthor{\bsnm{Miron}, \binits{R.}}:
\bbtitle{{F}insler-{L}agrange Geometry}.
\bpublisher{Editura Academiei Romane},
\blocation{Romania}
(\byear{2007})
\end{bbook}
\endbibitem

%%% 11
\bibitem[\protect\citeauthoryear{Bao et~al.}{2000}]{BCS}
\begin{bbook}
\bauthor{\bsnm{Bao}, \binits{D.}},
\bauthor{\bsnm{Chern}, \binits{S.-S.}},
\bauthor{\bsnm{Shen}, \binits{Z.}}:
\bbtitle{An Introduction to {R}iemann-{F}insler Geometry}.
\bpublisher{{Springer Graduate Texts in Mathematics Vol. 200}},
\blocation{New York}
(\byear{2000})
\end{bbook}
\endbibitem

%%% 12
\bibitem[\protect\citeauthoryear{Bejancu and Farran}{2000}]{BF}
\begin{bbook}
\bauthor{\bsnm{Bejancu}, \binits{A.}},
\bauthor{\bsnm{Farran}, \binits{H.R.}}:
\bbtitle{Geometry of Pseudo-{F}insler Submanifolds}.
\bpublisher{Springer},
\blocation{Dordrecht}
(\byear{2000})
\end{bbook}
\endbibitem

%%% 13
\bibitem[\protect\citeauthoryear{Minguzzi}{2014}]{Min}
\begin{barticle}
\bauthor{\bsnm{Minguzzi}, \binits{E.}}:
\batitle{Light cones in {F}insler spacetime}.
\bjtitle{Comm. {M}ath. {P}hys.}
\bvolume{334}(\bissue{3}),
\bfpage{1529}--\blpage{1551}
(\byear{2014})
\end{barticle}
\endbibitem

%%% 14
\bibitem[\protect\citeauthoryear{Dahl}{2006}]{Dahl}
\begin{botherref}
\oauthor{\bsnm{Dahl}, \binits{M.}}:
A brief introduction to {F}insler geometry.
\url{https://math.aalto.fi/~fdahl/finsler/index.html}.
Based on the author's licentiate thesis {\em Propagation of Gaussian beams
  using Riemann-Finsler geometry}, Helsinki University of Technology
(2006)
\end{botherref}
\endbibitem

%%% 15
\bibitem[\protect\citeauthoryear{Holm et~al.}{2009}]{HSS}
\begin{bbook}
\bauthor{\bsnm{Holm}, \binits{D.D.}},
\bauthor{\bsnm{Schmah}, \binits{T.}},
\bauthor{\bsnm{Stoica}, \binits{C.}}:
\bbtitle{Geometric Mechanics and Symmetry: from Finite to Infinite Dimensions}.
\bpublisher{Oxford University Press},
\blocation{Oxford}
(\byear{2009})
\end{bbook}
\endbibitem

%%% 16
\bibitem[\protect\citeauthoryear{Minguzzi}{2023}]{Min23}
\begin{barticle}
\bauthor{\bsnm{Minguzzi}, \binits{E.}}:
\batitle{A metrical approach to {F}insler geometry}.
\bjtitle{Rep. Math. Phys.}
\bvolume{92}(\bissue{2}),
\bfpage{173}--\blpage{195}
(\byear{2023})
\end{barticle}
\endbibitem

%%% 17
\bibitem[\protect\citeauthoryear{Vacaru}{2012}]{Vac}
\begin{barticle}
\bauthor{\bsnm{Vacaru}, \binits{S.I.}}:
\batitle{Principles of {E}instein-{F}insler gravity and persepctives in modern
  cosmology}.
\bjtitle{Int. J. Mod. Phys. B}
\bvolume{21}(\bissue{09}),
\bfpage{1250072}
(\byear{2012})
\end{barticle}
\endbibitem

%%% 18
\bibitem[\protect\citeauthoryear{Miron}{1986}]{Mir}
\begin{botherref}
\oauthor{\bsnm{Miron}, \binits{R.}}:
A {L}agrangian theory of relativity, {I}, {II}.
An. St. Univ. “Al. I.Cuza” Iasi
\textbf{{32 }}({7--16, 37--62})
(1986)
\end{botherref}
\endbibitem

%%% 19
\bibitem[\protect\citeauthoryear{Miron and Anastasiei}{1994}]{MiAn}
\begin{bbook}
\bauthor{\bsnm{Miron}, \binits{R.}},
\bauthor{\bsnm{Anastasiei}, \binits{M.}}:
\bbtitle{The Geometry of {L}agrange Spaces: Theory and Applications}.
\bpublisher{Kluwer Academic Publishers Group,},
\blocation{Dordrecht}
(\byear{1994})
\end{bbook}
\endbibitem

%%% 20
\bibitem[\protect\citeauthoryear{Miron and Kawaguchi}{1991}]{MK}
\begin{barticle}
\bauthor{\bsnm{Miron}, \binits{R.}},
\bauthor{\bsnm{Kawaguchi}, \binits{T.}}:
\batitle{Relativistic geometrical optics}.
\bjtitle{Int. J. Theor. Phys.}
\bvolume{{30}}(\bissue{11}),
\bfpage{1521}--\blpage{1543}
(\byear{1991})
\end{barticle}
\endbibitem

%%% 21
\bibitem[\protect\citeauthoryear{Antonelli et~al.}{1993}]{AIM}
\begin{bbook}
\bauthor{\bsnm{Antonelli}, \binits{P.L.}},
\bauthor{\bsnm{Ingarden}, \binits{R.S.}},
\bauthor{\bsnm{Matsumoto}, \binits{M.}}:
\bbtitle{The Theory of Sprays and Finsler Spaces with Applications in Physics
  and Biology}.
\bpublisher{{Springer Fundamental Theories of Physics vol. 58}},
\blocation{Dordretch}
(\byear{1993})
\end{bbook}
\endbibitem

%%% 22
\bibitem[\protect\citeauthoryear{Minguzzi}{2014}]{Min14}
\begin{barticle}
\bauthor{\bsnm{Minguzzi}, \binits{E.}}:
\batitle{The connections of pseudo-{F}insler spaces}.
\bjtitle{Int. J. Geom. Methods Mod. Phys.}
\bvolume{11}(\bissue{{No. 07}}),
\bfpage{1460025}
(\byear{2014})
\end{barticle}
\endbibitem

%%% 23
\bibitem[\protect\citeauthoryear{Kol\'a\v{r} et~al.}{2013}]{KMS}
\begin{bbook}
\bauthor{\bsnm{Kol\'a\v{r}}, \binits{I.}},
\bauthor{\bsnm{Michor}, \binits{P.W.}},
\bauthor{\bsnm{Slovák}, \binits{J.}}:
\bbtitle{Natural Operations in Differential Geometry}.
\bpublisher{Springer},
\blocation{Heidelberg}
(\byear{2013})
\end{bbook}
\endbibitem

%%% 24
\bibitem[\protect\citeauthoryear{Shen}{2001}]{Shen}
\begin{bbook}
\bauthor{\bsnm{Shen}, \binits{Z.}}:
\bbtitle{{Differential Geometry of Spray and Finsler Spaces}}.
\bpublisher{{Springer-Science+Business Media, B.V.}},
\blocation{Dordrecht}
(\byear{2001})
\end{bbook}
\endbibitem

%%% 25
\bibitem[\protect\citeauthoryear{Abate and Patrizio}{1994}]{AP}
\begin{bbook}
\bauthor{\bsnm{Abate}, \binits{M.}},
\bauthor{\bsnm{Patrizio}, \binits{G.}}:
\bbtitle{Finsler Metrics -- A Global Approach}.
\bpublisher{{Springer Lecture Notes in Mathematics vol. 1591}},
\blocation{{Scuola Normale Superiore, Pisa}}
(\byear{1994})
\end{bbook}
\endbibitem

%%% 26
\bibitem[\protect\citeauthoryear{Hohmann et~al.}{2022}]{HPV2}
\begin{barticle}
\bauthor{\bsnm{Hohmann}, \binits{M.}},
\bauthor{\bsnm{Pfeifer}, \binits{C.}},
\bauthor{\bsnm{Voicu}, \binits{N.}}:
\batitle{{Mathematical foundations for field theories on Finsler spacetimes}}.
\bjtitle{J. Math. Phys.}
\bvolume{63}(\bissue{3}),
\bfpage{032503}
(\byear{2022})
\end{barticle}
\endbibitem

%%% 27
\bibitem[\protect\citeauthoryear{Javaloyes and S\'anchez}{2019}]{JS19}
\begin{barticle}
\bauthor{\bsnm{Javaloyes}, \binits{M.A.}},
\bauthor{\bsnm{S\'anchez}, \binits{M.}}:
\batitle{On the definition and examples of cones and {F}insler spacetimes}.
\bjtitle{RACSAM}
\bvolume{114},
\bfpage{30}
(\byear{2019})
\end{barticle}
\endbibitem

\end{thebibliography}

\end{document}